\documentclass[12pt,twoside]{amsart}
\usepackage{amsmath}
\usepackage{amssymb}
\usepackage{adjustbox}
\usepackage{float}
\usepackage{hyperref}

\title[Fusion-invariant representations for symmetric groups]{Fusion-invariant representations for symmetric groups}

\author{Jos\'e Cantarero}
\thanks{}
\address{
\hfill\break Centro de Investigaci\'on en Matem\'aticas, A.C., Unidad M\'erida \\
\hfill\break Parque Cient\'ifico y Tecnol\'ogico de Yucat\'an \\ 
\hfill\break Carretera Sierra Papacal--Chuburn\'a Puerto Km 5.5 \\
\hfill\break Sierra Papacal, M\'erida, YUC 97302 \\
\hfill\break Mexico.}
\email{cantarero@cimat.mx}

\author{Jorge Gaspar-Lara}
\thanks{}
\address{
\hfill\break Centro de Investigaci\'on en Matem\'aticas, A.C. \\
\hfill\break C. Jalisco S/N, Col. Valenciana \\ 
\hfill\break Guanajuato, GTO 36023 \\
\hfill\break Mexico.}
\email{jorge.gaspar@cimat.mx}

\newcommand{\comments}[1]{}

\newcommand{\Aut}{\operatorname{Aut}\nolimits}

\newcommand{\Out}{\operatorname{Out}\nolimits}

\newcommand{\Ind}{\operatorname{Ind}\nolimits}

\newcommand{\Mod}[1]{\ \mathrm{mod}\ #1}

\newcommand{\reg}{\operatorname{reg}\nolimits}
\newcommand{\Rep}{\operatorname{Rep}\nolimits}
\newcommand{\Res}{\operatorname{Res}\nolimits}

\def \F{{\mathbb F}}

\def \Z{{\mathbb Z}}

\newcommand{\Ff}{{\mathcal{F}}}

\newcommand{\pcom}{^\wedge_p}

\theoremstyle{plain}
\newtheorem*{introtheorem}{Theorem}
\newtheorem{theorem}{Theorem}[section]
\newtheorem{proposition}[theorem]{Proposition}
\newtheorem{corollary}[theorem]{Corollary}
\newtheorem{lemma}[theorem]{Lemma}

\theoremstyle{definition}
\newtheorem{definition}[theorem]{Definition}
\newtheorem{remark}[theorem]{Remark}



\begin{document}

\begin{abstract}
For a prime $p$, we show that uniqueness of factorization into irreducible 
$\Sigma_{p^2}$--invariant representations of $\Z/p \wr \Z/p$ holds if and 
only if $p=2$. We also show nonuniqueness of factorization for $\Sigma_8$--invariant
representations of $D_8 \wr \Z/2$. The representation ring of $\Sigma_{p^2}$--invariant
representations of $\Z/p \wr \Z/p$ is determined completely when $p$ equals two
or three.
\end{abstract}

\maketitle


\section*{Introduction}

Decomposing objects into irreducible pieces is a common strategy used
in mathematics to better understand them. This philosophy can be traced
back to Democritus' atomic hypothesis, which affirms that matter is composed
of indivisible parts and the properties of each material depends on the
kinds of indivisible parts that it contains.

These decompositions become more useful when they are unique up to reordering 
and isomorphisms of the pieces, because they can be used to completely distinguish 
two objects up to isomorphism by comparing their decompositions. Results stating 
the existence and uniqueness of such decompositions are usually referred to as 
Krull-Schmidt theorems in different contexts, for instance in the representation
theory of finite groups. 

The theory of saturated fusion systems and $p$-local finite groups \cite{BLO} appears in modular representation
theory and algebraic topology to model certain objects which behave like a finite group in
some sense, such as blocks of the group ring with coefficients in a $p$-local ring or 
$p$-completed classifying spaces of finite groups. In particular, the $p$-local structure
of a finite group $G$, that is, the subgroups of a $p$-Sylow subgroup $S$ and $G$--subconjugations 
between them, is an example of a saturated fusion system over $S$. 

The study of complex representations of saturated fusion systems, also called fusion-invariant representations, was motivated by the study of
homotopy classes of maps from the classifying spaces of $p$-local finite groups to $p$-completed classifying
spaces of unitary groups in \cite{CCM}. In this paper we are only interested in the particular case of complex 
representations of a $p$-group $S$ whose characters are invariant under $G$--conjugacy, where $S$ is a $p$-Sylow
subgroup of $G$. We call them $G$--invariant representations of $S$.

The similarity with finite groups was reinforced in \cite{BC}. It was shown that the Grothendieck group $R_G(S)$ of $G$--invariant 
representations of $S$ is a free abelian group whose rank is the number of $G$--conjugacy classes of elements of $S$. 
There is also an analogue of the Atiyah-Segal completion theorem (see \cite{A}, \cite{AS}) relating $R_G(S)$ and the 
$K$-theory of $BG \pcom$.

However, we observe a different behaviour when we decompose $G$--invariant representations as a direct sum of irreducible 
$G$--invariant representations. Examples without uniqueness of decomposition were found in \cite{G} and \cite{Re}, and a 
framework to understand this phenomenon was developed in \cite{CCo}, where other examples were found using the software GAP. 
Uniqueness of decomposition holds if and only the number of $G$--conjugacy classes of $S$ equals the number of irreducible 
$G$--invariant representations of $S$. In this article we include the unpublished results of \cite{G} and expand them to find new examples
without uniqueness of decomposition.

We begin with a brief review of $G$--invariant representations and saturated fusion
systems in Section \ref{FusionRepresentations}. Identifying the wreath product $\Z/p \wr \Z/p$
with a $p$-Sylow subgroup of $\Sigma_{p^2}$, we describe explicit $\Sigma_{p^2}$--conjugations 
of subgroups of $\Z/p \wr \Z/p$ which, together with inclusions, generate the fusion system 
corresponding to the $p$-local structure of $\Sigma_{p^2}$. This description follows easily
from results in \cite{AD} and \cite{CLl}, but we include it for convenience in Section \ref{FusionSymmetric}. 
We use this description to determine the representation rings of $\Sigma_4$--invariant representations of $D_8$ and
of $\Sigma_9$--invariant representations of $\Z/3 \wr \Z/3$.

The representation ring of $\Sigma_9$--invariant representations of $\Z/3 \wr \Z/3$ can be used to determine
when a representation of $\Z/3 \wr \Z/3$ is $\Sigma_9$--invariant. In fact, we find four different irreducible $\Sigma_9$--invariant
representations $A$, $B$, $C$, $D$ of $\Z/3 \wr \Z/3$ which satisfy $A \oplus B = C \oplus D$, illustrating a nonunique
decomposition. Determining the representation ring of $\Sigma_{p^2}$--invariant representations of $\Z/p \wr \Z/p$ for 
all primes $p>3$, even as an abelian group, would be a messy endeavour. However, for each $p>3$, we found four irreducible $\Sigma_{p^2}$--invariant
representations of $\Z/p \wr \Z/p$ which are natural generalizations of the representations $A$, $B$, $C$ and $D$ of $\Z/3 \wr \Z/3$
and which determine a nonunique factorization. 

\begin{introtheorem}
Let $p$ be an odd prime. The $\Sigma_{p^2}$--invariant representations of $\Z/p \wr \Z/p$ do not satisfy uniqueness of factorization 
as a sum of $\Sigma_{p^2}$--invariant irreducible representations.
\end{introtheorem}

Since $\Sigma_4$--invariant representations of $D_8$ satisfy uniqueness of decomposition, it was natural to
study the case of $\Sigma_8$--invariant representations of $D_8 \wr \Z/2$, where a nonunique decomposition
was found using its character table. 

The phenomenon of nonuniqueness of decomposition of fusion-invariant representations is still not fully understood.
The characterizations in \cite{CCo} have a combinatorial nature which requires knowledge of all irreducible fusion-invariant
representations or a particular basis of the representation ring. It would be desirable to have a characterization
in terms of the fusion system only. With the computations in this article we contribute to the current list of examples
without uniqueness and we hope that they can be used to prove similar for $\Sigma_{p^k}$ with $p$ odd and $k>2$ and 
for $\Sigma_{2^k}$ with $k>3$.

\section{Fusion-invariant representations}
\label{FusionRepresentations}

In this section we provide a brief review of fusion-invariant representations
and their properties. We direct the reader to Section 1 of \cite{CCo} for further
aspects of these representations that we will not delve into.
\newline

Given an element $g$ in a group $G$, let $c_g(x)=gxg^{-1}$. The following definition
describes the main study subjects of this paper, also called fusion-invariant representations
when $G$ is clear from the context.

\begin{definition}
\label{FusionInvariantConcrete}
Let $G$ be a finite group and $S$ a $p$-Sylow subgroup of $G$. We say that
a representation $\rho \colon S \to U_n$ is $G$--invariant
if the representations $ \rho_{|P}$ and $\rho_{|gPg^{-1}} \circ c_g$ are equivalent
for any $P \leq S$ and any $ g \in G$ that satisfies $gPg^{-1} \leq S$.
\end{definition}

The motivation for this definition comes from the study of vector bundles, $K$-theory
and maps $BG \pcom \to BU(n) \pcom$. We direct the interested reader to \cite{BC} 
and \cite{CCM} for the connection with these problems. They have been studied more generally
for certain locally finite $p$-groups in \cite{CC}, \cite{NS}, \cite{Z1} and \cite{Z2}. 
This paper deals with algebraic properties of these representations, so familiarity with 
those papers is not necessary. They have been studied from this point of view in \cite{CCo},
\cite{G} and \cite{Re}.

By character theory, a representation $\rho$ is $G$--invariant if and only if its 
character $\chi$ satisfies
\[ \chi(x) = \chi_{\rho}(gxg^{-1}) \]
for any pair $(x,g) \in S \times G$ such that $gxg^{-1} \in S$. This definition
is a particular case of an $\Ff$--invariant representation of $S$, where $\Ff$ is a
saturated fusion system over $S$. It corresponds to the particular case of the
fusion system $\Ff_S(G)$. We will not consider $\Ff$--invariant representations
for fusion systems which are not of the form $\Ff_S(G)$. However, we will use the
terminology of centric and radical subgroups, corresponding to Definition 1.6 and 
Definition A.9 of \cite{BLO}, respectively.

\begin{definition}
Let $\Ff$ be a saturated fusion system over $S$. We say that a subgroup $P$ of $S$ 
is $\Ff$--centric if $C_S(Q) \leq Q$ for each $Q \leq S$ which is isomorphic to $P$
in the category $\Ff$.
\end{definition}

For $\Ff_S(G)$, this corresponds to a subgroup $P$ of $S$ that satisfies $C_S(Q) \leq Q$
for all the subgroups $Q$ of $S$ which are $G$--conjugate to $P$.

\begin{definition}
Let $\Ff$ be a saturated fusion system over $S$. We say that a subgroup $P$ of $S$ 
is $\Ff$--radical if $\Out_{\Ff}(P)$ has no nontrivial normal $p$-subgroups.
\end{definition}

In the case of $\Ff_S(G)$, we have $\Out_{\Ff}(P)=\Out_G(P)$. The importance of these 
subgroups comes from Alperin's fusion theorem (see Theorem A.10 in \cite{BLO} for instance)
which essentially states that a saturated fusion system is generated by automorphisms
of centric radical subgroups and inclusions. In particular, a representation $\rho$ of $S$ is
$G$--invariant if and only $\rho_{|P}$ and $\rho_{|P} \circ c_g$ are equivalent for any $\Ff_S(G)$--centric
radical subgroup $P$ of $S$ and any $g \in N_G(P)$.

The class of $G$--invariant representations is closed under equivalence, direct sums,
and tensor products. Therefore the set $\Rep_G(S)$ of equivalence classes of $G$--invariant representations 
of $S$ has the structure of a commutative semiring, and we denote its Grothendieck construction
by $R_G(S)$. If we ignore the multiplicative structure, Corollary 2.2 in \cite{BC} shows that this is
a free abelian group whose rank is the number of $G$--conjugacy classes of elements of $S$. We
determine these rings completely for $\Sigma_4$ and $\Sigma_9$ in Section \ref{Sigma4y9}. 

\begin{definition}
A $G$--invariant representation is called irreducible if it can not be decomposed
as the direct sum of two nontrivial $G$--invariant representations.
\end{definition}

An irreducible $G$--invariant representation may not be irreducible if we regard it as a 
representation of $S$. For instance, the representation $x$ in Example 6.1 in \cite{BC}
is an irreducible $\Sigma_p$--invariant representation of the $p$-Sylow $S$ of $\Sigma_p$,
but it is not irreducible as a representation of $S$. 

Perhaps more surprisingly, the decomposition of a $G$--invariant representation as a direct 
sum of irreducible $G$--invariant representations is not unique (as usual we identify reordered decompositions). 
For example, Example A.2 in \cite{Re} and Example 4.5.5 in \cite{G} illustrate this phenomenon for $PGL_3(\F_3)$ 
at the prime $2$ and for $\Sigma_9$ at the prime $3$, respectively. Other examples have been found in \cite{CCo} 
as well. We include in Sections \ref{FusionSymmetric}, \ref{Sigma4y9} and \ref{NonUnique} part of the content of 
\cite{G} and use it as a basis to find new examples of nonunique factorizations in Sections \ref{NonUnique} and
\ref{Sigma8}. In the language of \cite{CCo}, we are finding new examples of monoids of the form $\Rep_G(S)$ which are not
factorial. 

\section{Fusion in symmetric groups}
\label{FusionSymmetric}

In this section we recall some results about the fusion system of 
the symmetric group $\Sigma_{p^2}$ at the prime $p$.
\newline

For each $i \in \{0,\ldots ,p-1\}$, consider the $p$-cycles
\begin{gather*}
\sigma_i=(ip+1,\ldots ,(i+1)p), \\
\tau_i=(i+1,i+1+p,\ldots,i+1+p(p-1))
\end{gather*}
in $\Sigma_{p^2}$. The subgroup $S$ generated by $\sigma_0, \ldots, \sigma_{p-1}$ and
$\tau=\tau_0 \cdots \tau_{p-1}$ is a $p$-Sylow subgroup of $\Sigma_{p^2}$. It is clearly isomorphic to $\Z/p \wr \Z/p$. We let $K$ be the subgroup generated
by $\sigma_0, \ldots, \sigma_{p-1}$, which is elementary abelian of rank $p$, and let
$L$ be the subgroup generated by $\sigma_0\cdots \sigma_{p-1}$ and $\tau$, which is
elementary abelian of rank $2$. We also let $N$ be the subgroup generated by $\tau$.

\begin{lemma}
\label{ciclos}
Any element in $S-K$ is a $p^2$-cycle or a product of $p$ disjoint $p$-cycles.
\end{lemma}

\begin{proof}
Let $\rho \in S-K$. Since the order of $\rho$ is a power of $p$, it suffices to show that 
the action of $\rho$ on $\{1,\ldots,p^2\}$ has no fixed points. The element $\rho$ must
be of the form $\alpha \tau^b $ with $b \in \{ 1, \ldots, p-1 \}$ and $\alpha \in K$. 
Given $k \in \{ 1, \ldots, p^2 \}$, there exists a unique $i \in \{0,\ldots ,p-1\}$ such that $k\in \{ip+1,\ldots ,(i+1)p\}$. Note that
\[ \tau^b(k) \equiv k+bp \Mod p^2 . \]
Let $j\in \{0,\ldots ,p-1\}$ be such that $j \equiv i+b \Mod p$. Then $\tau^b(k)\in \{jp+1,\dots ,(j+1)p\}$
and $\rho(k)= \alpha \tau^b(k)\in \{jp+1,\dots ,(j+1)p\}$. Since $b\in \{1,\dots,p-1\}$, we have $j\neq i$ and
we conclude that $\rho(k)\neq k$.
\end{proof}

In the next lemma, we identify when each possibility happens. Given an element $ x = \sigma_0^{a_0}\cdots	\sigma_{p-1}^{a_{p-1}}\tau^a \in S$,
we let $\lambda(x) = a_0 + \ldots + a_{p-1}$ and $\mu(x) = a$.
		
\begin{lemma}
\label{conjS}
Two elements $x$, $y$ in $S-K$ are conjugate in $S$ if and only if $\mu(x)=\mu(y)$ and $\lambda(x) \equiv \lambda(y) \Mod p$. The element $x$ is
a product of $p$ different $p$-cycles if and only if $\lambda(x)$ is a multiple of $p$.
\end{lemma}

\begin{proof}
Note that
\[ \sigma_0^{c_0}\cdots \sigma_{p-1}^{c_{p-1}}\tau^c \sigma_0^{a_0}\cdots	\sigma_{p-1}^{a_{p-1}}\tau^a \tau^{-c}\sigma_0^{-c_0}\cdots \sigma_{p-1}^{-c_{p-1}} =  \sigma_c^{a_0+c_c-c_{c-a}}\cdots	\sigma_{p-1 +c}^{a_{p-1}+c_{p-1+c}-c_{p-1+c-a}} \tau^a , \]
from where we see that $\mu(x)$ and the residue of $\lambda(x)$ modulo $p$ are $S$--conjugacy invariants. On the other hand, let 
$ x = \sigma_0^{a_0}\cdots	\sigma_{p-1}^{a_{p-1}}\tau^a$ and $ y= \sigma_0^{b_0}\cdots \sigma_{p-1}^{b_{p-1}}\tau^b$ be such that 
$a=b$ and $\lambda(x) \equiv \lambda(y) \Mod p$. We can rewrite the last congruence as
\[ a_0- \left( \sum_{i=1}^{p-1} b_{ia}-a_{ia} \right) \equiv b_0 \Mod p . \]
Consider the element
\[ z = \sigma_a^{b_a-a_a} \sigma_{2a}^{b_{2a}-a_{2a}+b_a-a_a}\cdots \sigma_{(p-1)a}^{\sum_{i=1}^{p-1} b_{ia}-a_{ia}} . \]
A straightforward computation using the congruence above shows that
\[ zxz^{-1} = y \]
and this shows the first claim. For the second claim, note that if $\lambda(x)$ is a multiple of $p$, then $x$ is conjugate in $S$
to a nontrivial element of $N$, hence it is a product of $p$ different $p$-cycles. On the other hand, if $ x = \sigma_0^{a_0}\cdots	\sigma_{p-1}^{a_{p-1}}\tau^a$ 
is a product of $p$ different $p$-cycles, then
\[ 1 = x^p = \sigma_0^{a_0+\ldots+a_{p-1}} \cdots \sigma_{p-1}^{a_0+\ldots+a_{p-1}}, \]
from where $\lambda(x)$ is a multiple of $p$.
\end{proof}

Recall the following notation and terminology from Definition 7 of \cite{AD}. Let $p$ 
be a prime and let $c_1, \ldots ,c_t,d$ be positive integers. For any enumeration of the
elements of $(\Z/p)^d$, the corresponding permutation action of $(\Z/p)^d$ on $\{1,\ldots,p^d\}$ 
determines a monomorphism $ (\Z/p)^d \to \Sigma_{p^d}$ and we denote by $A_d$ the image of
this map. Different enumerations give rise to conjugate subgroups.

The basic subgroup of $\Sigma_{p^{c_1+\ldots+c_t}}$ associated to $\textbf{c}=(c_1,\dots c_t)$ is given by
\[ A_{\textbf{c}}=A_{c_1}\wr \cdots \wr A_{c_t} . \]
We use the notation $|\textbf{c}|=c_1+\ldots+c_t$. We now identify the possible 
centric radical subgroups in $\Ff_S(\Sigma_{p^2})$ and their automorphism groups. 

\begin{proposition}
If $p$ is an odd prime, the only centric radical subgroups in $\Ff_S(\Sigma_{p^2})$ up to isomorphism are $S$, $K$ and $L$.
The only centric radical subgroups in $\Ff_S(\Sigma_4)$ up to isomorphism are $S$ and $L$.
\end{proposition}

\begin{proof}
Let $R \leq S$ be a centric radical subgroup in $\Ff_S(\Sigma_{p^2})$. Since $R$ is 
radical, by Theorem 9 from \cite{AD}, there exist decompositions
\begin{gather*}
\{1, \ldots, p^2 \} = V_0 \amalg \ldots \amalg V_s , \\
R = \{ 1 \} \times R_1 \times \cdots \times R_s ,
\end{gather*}
where $R_i$ is a basic subgroup of $S_{V_i}$ associated to $\textbf{c}_i$ and $| V_i | =p^{|\textbf{c}_i|}$ for $i\in \{1,\ldots ,s\}$.
Since $|V_i| \leq p^2$, we have $|\textbf{c}_i|\leq 2$. 

If there exists $j$ with $|\textbf{c}_j|=2$, then $|V_j|=p^2$. Hence $s=j=1$ and $V_0$ is empty. If $\textbf{c}_1=(1,1)$, then $R=A_1\wr A_1 = S$.
If $\textbf{c}_1=(2)$, then $R=A_2$. If we enumerate $(\Z/p)^2$ using the bijection
\begin{align*}
\varphi \colon (\Z/p)^2 & \to \{1,\ldots,p^2\},  \\
(l,k) & \mapsto  lp+(k+1), 
\end{align*}
the image of the corresponding permutation action $(\Z/p)^2 \to \Sigma_{p^2}$ is precisely $L$. By Theorem 9 in \cite{AD}, the centralizer
of $A_2$ in $\Sigma_{p^2}$ is contained in $A_2$, hence $C_S(gA_2g^{-1}) \leq gA_2 g^{-1}$ for any $g \in G$. That is, $L$ is centric. Since
$L$ is abelian and centric, we have
\[ \Out_{\Sigma_{p^2}}(L) = N_{\Sigma_p^2}(L)/L \]
and by part (b) of \cite[Theorem 9]{AD}, this group is isomorphic to $GL_2(\F_p)$. The order of this group is $(p^2-1)(p-1)p$, so the
order of its maximal normal $p$-subgroup divides $p$. It it had order $p$, it would be a normal $p$-Sylow subgroup, but the subgroups
generated by the matrices
\[ \left(\begin{array}{cc}  1 & 0 \\ 1 & 1 \end{array}\right)  \text{ and } \left(\begin{array}{cc}  1 & 1 \\ 0 & 1 \end{array}\right) \]
are two different $p$-Sylow subgroups of $GL_2(\F_p)$. Therefore $L$ is radical.

On the other hand, if $|c_i|=1$ for each $i \in \{1,\ldots ,s\}$, then $|V_i|=p$ and $R_i=A_1\cong \Z/p$ for $ i \geq 1$. Each $R_i$
is generated by a $p$-cycle. If $s<p$, then $|R|=ps<p^2$, so there exists a set $J=\{j_1,\ldots , j_p\}$ such that $J\cap V_i=\emptyset$ for all $i \in \{1,\ldots ,s\}$.
Then $\sigma=(j_1,\ldots ,j_p) \in C_S(R) - R$, so the subgroup $R$ is not centric.

If $s=p$, then $R$ is a product of $p$ cyclic groups, each generated by a $p$-cycle. By
Lemma \ref{ciclos}, we must have $R=K$. In particular, $K$ is the only subgroup of $S$ 
in its $G$--conjugacy class. Hence it suffices to show $C_S(K) \leq K$ to prove that $K$ 
is centric. Given an element $\rho=\sigma_0^{a_0}\cdots	\sigma_{p-1}^{a_{p-1}}\tau^b \in S\setminus K$ we have
\begin{gather*}
\sigma_0 \rho =\sigma_0^{a_0 +1}\cdots	\sigma_{p-1}^{a_{p-1}}\tau^b, \\
\rho \sigma_0=\sigma_0^{a_0} \cdots \sigma_b ^{a_b +1} \cdots	\sigma_{p-1}^{a_{p-1}}\tau^b .
\end{gather*}
Since $b$ is not a multiple of $p$, we can conclude $\rho \sigma_0\neq \sigma_0 \rho$. Using part (b) of Theorem 9 in \cite{AD}, we have
\[ \Out_{\Sigma_{p^2}}(K) \cong GL_1(\F_p) \wr \Sigma_p = \F_p^{\times} \wr \Sigma_p . \]
If $p=2$, this is isomorphic to $\Z/2$, so $K$ is not radical. Assume now $p \neq 2$. The order of this group is $(p-1)^p p!$, 
so the order of its maximal normal $p$-subgroup divides $p$. The element $((1,\ldots,1),(1,2,\ldots ,p))$ has order $p$
and generates a $p$-Sylow subgroup of $\F_p^{\times} \wr \Sigma_p$. Let $q$ be a generator of $\F_p^{\times}$. Then 
\[ ((q,\ldots,1),(1))((1,\ldots,1),(1,2,\ldots ,p))((q^{-1},\ldots,1),(1)) =((q,q^{-1},\ldots,1),(1,2,\ldots ,p)) \]
is an element of orden $p$ which generates a different $p$-Sylow subgroup. Hence $K$ is radical. Since $S$ is always 
centric radical in any saturated fusion system over $S$, this concludes the proof.
\end{proof}

In the previous proof we identified the outer automorphism group of $K$ and $L$ in the fusion system of $\Sigma_{p^2}$.
It will be convenient to give explicit generators in each case, as well as in the case of $S$. Let $q$ be a generator
of $\F_p^{\times}$. First we have
\[ \Out_{\Sigma_{p^2}}(L) \cong GL_2(\F_p). \]
By Theorem 5.1 in \cite{MGL}, this group is generated by the matrices
\[ \left( \begin{array}{cc}
   1 & 0 \\ 
  -1 & -1 \end{array}\right), \qquad	\left(\begin{array}{cc}
                                            0 & 1 \\ 
                                            1 & 0 \end{array}\right), \qquad \left(\begin{array}{cc}  
                                                                                    q & 0 \\ 
                                                                                    0 & 1 \end{array}\right). \]
Second, for $K$ we have 
\[ \Out_{\Sigma_{p^2}}(K) \cong \F_p^{\times} \wr \Sigma_p ,\]
and it is generated by the elements
\[ ((q,1, \ldots,1),1),((1,q,\ldots,1),1),\ldots , ((1,1,\ldots,q),1) \]
and
\[ ((1,\ldots,1),(1,2)), ((1,\ldots,1),(1,2,\ldots ,p)). \]
Finally, we have 
\[ \Out_{\Sigma_{p^2}}(S) \cong GL_1(\F_p) \times GL_1(\F_p) = \F_p^{\times} \oplus \F_p^{\times} \]
which clearly has generators $(q,1)$ and $(1,q)$. Elements of $N_{\Sigma_{p^2}}(S)$ representing them
can be obtained from \cite{CLl}. Namely, a splitting of the short exact sequence 
\[ 0 \to S \to N_{\Sigma_{p^2}}(S) \to \F_p^{\times} \oplus \F_p^{\times} \to 0 \]
is given by 
\begin{align*}
 & \psi \colon \F_p^{\times} \oplus \F_p^{\times} \to N_{\Sigma_{p^2}}(S), \\
 & \psi(q,1)(lp+(k+1))=qlp+(k+1), \\
 & \psi(1,q)(lp+(k+1))=lp+(qk+1), 
\end{align*}
for $l\in \{0,\ldots ,p-1\}$ and $k\in \{0,\ldots ,p-1\}$. The expressions $lp+(k+1)$ and their
images are taken modulo $p^2$.

\section{Fusion-invariant representations for $\Sigma_4$ and $\Sigma_9$}
\label{Sigma4y9}

In this section we determine the representation rings of the fusion systems of $\Sigma_4$
and $\Sigma_9$ at the primes $2$ and $3$, respectively. Then we find a nonunique factorization
into irreducible $\Sigma_9$--invariant representations of $\Z/3 \wr \Z/3$. From now on, we use
the same notation for a representation and its equivalence class. We also use the symbol $+$
for the direct sum of representations or for the equivalence class of the direct sum of the
representatives.
\newline

We begin with $\Sigma_4$ at the prime $2$. The $2$-Sylow subgroup described in the previous section is generated by
$(1,2)$, $(3,4)$ and $(1,3)(2,4)$. It is also generated by the elements $(1,4,2,3)$ and $(1,2)$, and we obtain an isomorphism
between $S$ and $D_8$ by identifying them with the usual generators $r$ and $s$. Its subgroup $L$ is generated by 
\[ (1,2)(3,4), (1,3)(2,4) \]
and $\Out_{\Sigma_4}(L)$ is generated by $c_{(23)}$ and $c_{(24)}$, while $\Out_{\Sigma_4}(S)$ is trivial. We will identify
$S$ with $D_8$ and $L$ with the subgroup $\{1,r^2,sr,sr^3 \}$ for convenience. Note that $c_{(2,3)}$ permutes $r^2$ with $sr^3$
and $c_{(2,4)}$ permutes $r^2$ with $sr$, both of them fixing the remaining elements of $L$.

Let $X$, $Y$ and $XY$ be the nontrivial irreducible $1$-dimensional representations of $S$ whose kernels are the subgroups
generated by $\{r\}$, $\{s,r^2\}$ and $\{sr^3,r^2\}$, respectively. Let $Z$ be its irreducible two-dimensional representation.
Similarly, let $V$, $W$ and $VW$ be the nontrivial irreducible one-dimensional representations of $L$ whose kernels are the
subgroups generated by $r^2$, $sr$ and $sr^3$, respectively. Note that the notations $XY$ and $VW$ are used because these
representations are in fact tensor products. The following table can be obtained directly or using characters.

\begin{table}[H]
\begin{center}
\begin{tabular}{l|lll}
$A$ & $\Res^S_L(A)$ & $c_{(2,3)}^*\Res^S_L(A)$ & $c_{(2,4)}^*\Res^S_L(A)$ \\ \hline
$X$   		& $V$ & $VW$ & $W$ \\ 
$Y$ 	    & $V$ & $VW$ & $W$ \\
$Z$ 		& $W+VW$ & $W+V$ & $V+VW$ \\   
\end{tabular}
\caption{Restrictions of irreducible representations of $D_8$}
\end{center}
\end{table}

Now a virtual representation
\[ \psi=a1+bX+cY+dXY+eZ \]
of $S$ is $\Sigma_{p^2}$--invariant if and only if both $c_{(2,3)}^*\Res^S_L(\psi)$ and $c_{(2,4)}^*\Res^S_L(\psi)$ equal $\Res^S_L(\psi)$, that is,
\begin{gather*}
a1+bVW+cVW+d1+e(W+V) = a1+bV+cV+d1+e(W+VW), \\
a1+bW+cW+d1+e(V+VW) = a1+bV+cV+d1+e(W+VW),
\end{gather*}
from where $ e = b+c$. Therefore the set $\{1,XY,X+Z,Y+Z\}$ is a basis of $R_{\Sigma_4}(D_8)$ as a free abelian group.
Using character theory, it is easy to determine that, as a ring,
\[ R_{\Sigma_4}(D_8) \cong  \Z[r,s,t]/(r^2-r-s-t-2,s^2-r^2,rs-r-s-2t,rt-t,st-r,t^2-1)\] 
where $r$, $s$ and $t$ correspond to $X+Z$, $Y+Z$ and $XY$, respectively.
\newline

Now we treat the case of $\Sigma_9$ at the prime $3$. In this case its $3$-Sylow subgroup $S$ is generated by
the elements
\[ (1,2,3) , (4,5,6), (7,8,9) , (1,4,7)(2,5,8)(3,6,9). \]
The irreducible representations of $S$ can be determined by the method of Section 8.2 in \cite{SJLR},
since $S = K \rtimes \Z/3$. Recall that $K$ is an elementary abelian $3$-group of rank $3$ generated
by the elements $\sigma_0, \sigma_1, \sigma_2$ and let $\omega=e^{2\pi i /3}$. We denote by $\omega^j$
the irreducible representation of $\Z/3$ that sends $[1]_3$ to $\omega^j$. The irreducible representations
of $K$ have the form
\[ \omega^{j_0} \otimes \omega^{j_1} \otimes \omega^{j_2}, \]
where $ 0 \leq j_k \leq 2$. The group $\Z/3$ acts on the representations of $K$ by cyclic permutations of the
superindices $j_k$, and the fixed points of the action have the form $ \omega^j \otimes \omega^j \otimes \omega^j$.
Each representation $ \omega^j \otimes \omega^j \otimes \omega^j$ extends to $S$ and we tensor it with the inflation to $S$ of an irreducible representation $\rho$ 
of $\Z/3$. These are some of the irreducible representations of $S$, which we denote by
\[ \omega^j \otimes \omega^j \otimes \omega^j \otimes \rho. \]
Each irreducible representation $\alpha$ of $K$ which is not fixed by the action of $\Z/3$ gives rise to another irreducible
representation of $S$, namely its induction $\Ind_K^S(\alpha)$. In the appendix we show the character table of $S$, which uses
the following short names for the irreducible representations.
\begin{align*}
a_j & = 1 \otimes 1 \otimes 1 \otimes \omega^j ,\\
b_j & = \omega \otimes \omega\otimes \omega \otimes \omega^j ,\\
c_j & = \omega^2\otimes \omega^2\otimes \omega^2\otimes \omega^j, \\
x_j & = \Ind_K^S( \omega^j\otimes 1\otimes 1 ) ,\\
y_0 & = \Ind_K^S( \omega\otimes \omega \otimes 1), \\
y_1 & = \Ind_K^S( \omega\otimes \omega^2 \otimes 1), \\
y_2 & = \Ind_K^S( \omega^2\otimes \omega \otimes 1), \\
y_3 & = \Ind_K^S( \omega^2\otimes \omega^2 \otimes 1), \\
z_0 & = \Ind_K^S( \omega \otimes \omega \otimes \omega^2), \\
z_1 & =\Ind_K^S( \omega \otimes \omega^2 \otimes \omega^2). 
\end{align*}
Note that $a_0$ is the trivial one-dimensional representation, hence we will
denote it by $1$. The group $\Out_{\Sigma_9}(S)$ is generated by 
\[ c_{(23)(56)(89)}, c_{(47)(58)(69)} \]
and it is easy to find a basis of $R(S)^{\Out_{\Sigma_9}(S)}$ by taking the
sums over the orbits of each irreducible representation of $S$ for the
action of $\Out_{\Sigma_9}(S)$. This basis is
\[ \{ 1, a_1+a_2, b_0+c_0,b_1+b_2+c_1+c_2,x_0+x_1,y_0+y_3,y_1+y_2,z_0+z_1 \}. \]
Let $v_j$ be the $j$th element in this basis for $ j = 1, \ldots, 8$. The restrictions 
to $K$ of these representations are given in the following table.

\begin{table}[H]
\begin{center}
	\begin{tabular}{l|l}
 $A$ & $\Res_K^S(A)$  \\ \hline
$v_1$ 	& $1$ \\
$v_2$	& $2$ \\ 
$v_3$	& $\omega\otimes \omega\otimes \omega + \omega^2\otimes \omega^2\otimes \omega^2$ \\
$v_4$	& $2\omega\otimes \omega\otimes \omega+2\omega^2\otimes \omega^2\otimes \omega^2$ \\ 
$v_5$ & $\omega\otimes 1\otimes 1 + 1\otimes \omega\otimes 1 + 1\otimes 1\otimes \omega+\omega^2\otimes 1\otimes 1 + 1\otimes \omega^2\otimes 1 +1\otimes 1\otimes \omega^2$ \\
$v_6$	& $\omega\otimes \omega \otimes 1 + 1\otimes \omega \otimes \omega + \omega\otimes 1\otimes \omega + \omega^2 \otimes \omega^2 \otimes 1 + 1 \otimes \omega^2 \otimes \omega^2  +\omega^2 \otimes 1 \otimes \omega^2$ \\
$v_7$	& $\omega\otimes \omega^2 \otimes 1 + 1\otimes \omega\otimes \omega^2  +\omega^2 \otimes 1 \otimes \omega + \omega^2\otimes \omega \otimes 1 +1\otimes \omega^2 \otimes \omega +\omega\otimes 1 \otimes \omega^2$ \\
$v_8$	& $ \omega \otimes \omega \otimes \omega^2 + \omega^2 \otimes \omega \otimes \omega + \omega \otimes \omega^2 \otimes \omega+\omega \otimes \omega^2 \otimes \omega^2 + \omega^2 \otimes \omega \otimes \omega^2 + \omega^2 \otimes \omega^2 \otimes \omega$ 
	\end{tabular}
	\caption{Restrictions of representations $v_j$ to $K$}
\end{center} 
\end{table}

The group $\Out_{\Sigma_9}(K)$ is generated by
\[ c_{(23)}, c_{(56)}, c_{(89)}, c_{(14)(25)(36)}, c_{(147)(258)(369)}, \]
but we can ignore the last element because $(1,4,7)(2,5,8)(3,6,9)\in S$.
Let $ v = \sum t_j v_j$ be a $\Sigma_9$--invariant virtual representation. Then
it must satisfy $\Res_K^S(v) = c_{(2,3)}^*\Res_K^S(v)$. Since $v_1$, $v_2$
and $v_5$ satisfy this property, we focus on the rest. For instance, the
coefficient of $\omega^2 \otimes \omega \otimes \omega$ in $\Res_K^S(v)$
is $t_8$ and in $c_{(2,3)}\Res_K^S(v)$ is $t_3+2t_4$, hence
\[ t_3+2t_4 = t_8 . \]
Similarly, the coefficients of $\omega \otimes \omega \otimes 1$ are $t_6$
and $t_7$, hence $t_6=t_7$. By symmetry, considering the coefficients of
the other irreducible representations of $K$ that appear in $\Res_K^S(v_i)$
for $i=3,4,6,7,8$ do not impose new conditions. Therefore $v$ is an integral 
linear combination of the elements
\[ v_1, v_2, v_3+v_8, v_4+2v_8, v_5, v_6+v_7 . \]
By symmetry again, the automorphisms $c_{(56)}$ and $c_{(89)}$ give rise to
the same conditions on the coefficients $t_j$. The automorphism $c_{(14)(25)(36)}$
permutes $\sigma_0$ and $\sigma_1$, and we see clearly that $\Res_K^S(v_j) = c_{(147)(258)(369)}^*\Res_K^S(v_j)$
for all $j$. 

The restrictions to $L$ of the representations we found above are given in the next table.

\begin{table}[H]
\begin{center}
	\begin{tabular}{l|l}
 $A$ & $\Res_L^S(A)$  \\ \hline
$v_1$ 	& $1$ \\
$v_2$	& $1 \otimes \omega + 1 \otimes \omega^2$ \\ 
$v_3+v_8$	& $ 2 + \omega\otimes 1 + \omega\otimes \omega + \omega\otimes \omega^2 + \omega^2\otimes 1 + \omega^2\otimes \omega + \omega^2\otimes \omega^2$ \\
$v_4+2v_8$	& $ 2(1 \otimes \omega + 1 \otimes \omega^2 + \omega\otimes 1 + \omega\otimes \omega + \omega\otimes \omega^2 + \omega^2\otimes 1 + \omega^2\otimes \omega + \omega^2\otimes \omega^2)$ \\ 
$v_5$ & $\omega\otimes 1 + \omega\otimes \omega + \omega\otimes \omega^2 + \omega^2\otimes 1 + \omega^2\otimes \omega + \omega^2\otimes \omega^2$ \\
$v_6+v_7$	& $\omega\otimes 1 + \omega\otimes \omega + \omega\otimes \omega^2 + 2(1+ 1\otimes \omega + 1\otimes \omega^2) + \omega^2\otimes 1 + \omega^2\otimes \omega + \omega^2\otimes \omega^2$ \\
	\end{tabular}
	\caption{Restrictions of representations to $L$}
\end{center} 
\end{table}

The group $\Out_{\Sigma_9}(L)$ is generated by
\[ c_{(23)(56)(89)}, c_{(24)(73)(68)}, c_{(17)(25)(69)}. \]
Let $\rho_i$ be the $i$th representation in the previous table. Since the automorphism
$c_{(23)(56)(89)}$ also defines an element of $\Aut_{\Sigma_9}(S)$, we have 
$\Res_L^S(\rho_i) = c_{(23)(56)(89)}^*\Res_L^S(\rho_i)$ for all $i$. The automorphism
$c_{(24)(73)(68)}$ permutes $\sigma$ and $\tau$. Let $\rho = \sum s_j \rho_j$ be a $\Sigma_9$--invariant
virtual representation. The coefficient of $1 \otimes \omega $ in $\Res_L^S(\rho)$ is
\[ s_2+2s_4+2s_6 , \]
while the coefficient of $1 \otimes \omega $ in $c_{(24)(73)(68)}^*\Res_L^S(\rho)$ is
\[ s_3 + 2s_4 +s_5+s_6 . \]
Hence we must have
\[ s_3 + 2s_4 +s_5+s_6 = s_2+2s_4+2s_6 , \]
from where $ s_2 = s_3+s_5-s_6$. The coefficients of the other representations give us no new restrictions because these representations
are either invariant under the permutation of $\sigma$ and $\tau$, or because $\rho_i$ is invariant under
the transformations that square $\sigma$ or $\tau$. Therefore any $\Sigma_9$--invariant virtual representation
is an integral linear combination of 
\[ v_1, v_2+v_3+v_8, v_4+2v_8, v_2+v_5, v_6+v_7-v_2 . \]
Note that the restrictions of these representations to $L$ are 
\[ 1, \reg_L+1,2(\reg_L-1),\reg_L-1,\reg_L+1 , \]
where $\reg_L$ is the regular representation of $L$. These
representations are invariant with respect to any automorphism of $L$,
in particular with respect to $c_{(17)(25)(69)}^*$. Therefore we have 
shown that
\[ \{ v_1, v_2+v_3+v_8, v_4+2v_8, v_2+v_5, v_6+v_7-v_2 \} \]
is a basis for $R_{\Sigma_9}(S)$ as a free abelian group. If we let $X$, $Y$ and $Z$ be the 
sum of the representations $x_j$, $y_j$ and $z_j$, respectively, this basis is 
\[ \{ 1, a_1+a_2+b_0+c_0+Z,b_1+b_2+c_1+c_2+2Z,a_1+a_2+X,Y-a_1-a_2 \}. \]
But we will work with the alternative basis
\[ \{ 1, P = a_1+a_2+X, Q = b_0+c_0+Z-X, R=b_1+b_2+c_1+c_2+2Z, T=X+Y \} \]
for convenience. Using the character table for $S$, it is easy to determine
the multiplicative structure of the representation ring, namely
\[ R_{\Sigma_9}(S) \cong \Z[P,Q,R,T]/J , \] 
where $J$ is the ideal generated by the elements 
\[ \left. \begin{array}{ll}
P^2 -4-3P-2T, & \qquad \qquad PQ + 2 +2P-Q-2R ,\\
PR -4R-3Q-4T ,& \qquad \qquad PT -2-2P-Q-R-6T, \\
Q^2 -6-4P+Q+2R, & \qquad \qquad QR -4-6P+2Q+R ,\\
QT+2P-Q-R-2T ,& \qquad \qquad R^2 -12-10P-2Q-R-8T, \\
RT-4Q-4R-12T ,& \qquad \qquad T^2-6-6P-4Q-4R-11T.
\end{array} \right. \]
%

\section{Non-unique decompositions for $\Sigma_{p^2}$}
\label{NonUnique}

In the rest of the article, when we say that uniqueness of decomposition into irreducible $G$--invariant 
representations of $S$ is not satisfied, we mean that there exist at least two decompositions which are different
even after reordering. In the language of \cite{CCo}, this says that the monoid $\Rep_G(S)$ is not factorial. In 
this section we show that the $\Sigma_{p^2}$--invariant representations of $\Z/p \wr \Z/p$ do not satisfy uniqueness 
of decomposition into irreducible $\Sigma_{p^2}$--invariant representations if $p$ is an odd prime.
\newline

In the previous section we determined a basis $B=\{1,P,Q,R,T\}$ of $\Sigma_9$--invariant virtual representations of the $3$-Sylow $S$
of $\Sigma_9$. We recall their decomposition as a sum of irreducible representations of $S$ in the next table.

\begin{table}[H]
\begin{center}
	\begin{tabular}{l|l}
 &   \\ \hline
$1$ 	& $1$ \\
$P$	& $a_1+a_2+x_0+x_1$ \\ 
$Q$	& $b_0+c_0+z_0+z_1-x_0-x_1$ \\
$R$	& $b_1+b_2+c_1+c_2+2z_0+2z_1$ \\ 
$T$ & $x_0+x_1+y_0+y_1+y_2+y_3$
\end{tabular}
\caption{Decompositions of the elements of the basis of $R_{\Sigma_9}(\Z/3 \wr \Z/3)$}
\end{center}
\end{table}

Note that $P+Q+T$ is a $\Sigma_9$--invariant representation of $S$ which can be expressed as
a sum of $\Sigma_9$--invariant representations in two ways:
\[ P+(Q+T) = (P+Q)+ T \]
Let us show that the representations $P$, $Q+T$, $P+Q$ and $T$ are irreducible $\Sigma_9$--invariant representations.
Let $V = \sum_{h \in B} m_h h$ be a $\Sigma_9$--invariant subrepresentation of $P$. Since $P$ does not contain $1$, 
$y_j$, $c_j$ as subrepresentations, and each of these representations only appear in one of the elements of $B$, the coefficients $m_1$, $m_R$, $m_Q$
and $m_T$ vanish. Therefore $V = m_P P$ and since $V$ is a subrepresentation of $P$, we must have $V=0$ or $V=P$. The argument for $T$ is similar, 
considering that $T$ does not contain $1$, $a_j$, $b_j$ nor $c_j$ as subrepresentations, and each of these representations only appear in one of the elements of $B$.
Therefore $P$ and $T$ are irreducible $\Sigma_9$--invariant representations.

Let $V = \sum_{h \in B} m_h h$ be a $\Sigma_9$--invariant subrepresentation of $Q+T$. Since $Q+T$ does not contain $1$, $a_j$ nor $c_j$ as
subrepresentations, and each of these representations only appear in one of the elements of $B$, the coefficients $m_1$, $m_P$ and $m_R$
vanish. Hence $V = m_Q Q + m_T T$. Since $V$ is a subrepresentation of $Q+T$, the coefficient of $b_0$, which is $m_Q$, must be $0$ or $1$.
Similarly, the coefficient of $y_0$, which is $m_T$, must be $0$ or $1$. The possibility $V=Q$ can not happen because $V$ is a representation
and the possibility $V=T$ can not happen because $T$ is not a subrepresentation of $Q+T$. The argument for $P+Q$ is similar, considering that
$P+Q$ does not contain $1$, $b_1$ nor $y_1$ as subrepresentations, and each of these representations only appear in one of the elements of $B$.
And we can discard the possibilities $P$ and $Q$, because $P$ is not a subrepresentation of $P+Q$ and $Q$ is not a representation. Therefore
$Q+T$ and $P+Q$ are irreducible $\Sigma_9$--invariant representations.

\begin{remark}
Remark 2.12 in \cite{CCo} contains an irreducibility criterion for fusion-invariant representations
which could have been used here. Indeed, Example 2.5 in \cite{CCo} shows the irreducibility of $P$,
$Q+T$, $P+Q$ and $T$ using this criterion. 
\end{remark}

\begin{proposition}
\label{NonUniqueSigma9}
The $\Sigma_9$--invariant representations of $\Z/3 \wr \Z/3$ do not satisfy uniqueness of factorization 
as a sum of irreducible $\Sigma_9$--invariant representations.
\end{proposition}

We will now generalize this to $\Sigma_{p^2}$ where $p$ is an odd prime. Unlike the case $p=3$, we
will not determine the representation ring, but only describe a nonunique factorization inspired by
the factorization found in Proposition \ref{NonUniqueSigma9}. Let $S$ be the $p$-Sylow subgroup of $\Sigma_{p^2}$
introduced in Section \ref{FusionSymmetric}, which is a semidirect product of a rank-$p$ elementary
abelian $p$-group $K$ by $\Z/p$.

Let $\omega = e^{2\pi i/p}$. We denote by $\omega^j$ the irreducible representation of $\Z/p$ that sends $[1]_p$ to $\omega^j$. 
The irreducible representations of $K$ have the form
\[ \omega^{j_0} \otimes \ldots \otimes \omega^{j_{p-1}}, \]
where $ 0 \leq j_k \leq p-1$. The group $\Z/p$ acts on the representations of $K$ by cyclic permutations of the
superindices $j_k$, and the fixed points of the action have the form $ \omega^j \otimes \ldots \otimes \omega^j$.
Eachs representation $ \omega^j \otimes \ldots \otimes \omega^j$ extends to $S$ and we tensor it with the inflation to $S$ of an irreducible representation $\rho$ 
of $\Z/p$. These are some of the irreducible representations of $S$, which we denote by
\[ \omega^j \otimes \ldots \otimes \omega^j \otimes \rho . \]
Each irreducible representation $\alpha$ of $K$ which is not fixed by the action of $\Z/p$ gives rise to another irreducible
representation of $S$, namely its induction $\Ind_K^S(\alpha)$. We will use the following short names:
\begin{align*}
a_i^j & = \omega^j \otimes \ldots \otimes \omega^j \otimes \omega^i , \\
X_i & = \Ind_K^S \left( \sum_{(j_0,\ldots,j_{p-1}) \in B_i} \omega^{j_0} \otimes \ldots \otimes \omega^{j_{p-1}} \right) .
\end{align*}
For $ 1 \leq i < p$, we are considering a set $B_i$ of representatives of the $\Z/p$--orbits of the set $W_i$ of $p$-tuples 
$(j_0,\ldots,j_{p-1}) \in \{0,\ldots,p-1\}^p$ with exactly $i$ nonzero coordinates, under the cyclic permutation action of $\Z/p$. 
For $i=p$, we are taking a set $B_p$ of representatives of the $\Z/p$--orbits of the set
$W_p$ of $p$-tuples $(j_0,\ldots,j_{p-1}) \in \{1,\ldots,p-1\}^p-\{(j,\ldots,j)\}$.

In the following lemmas, we show that certain representations of $S$ are irreducible $\Sigma_{p^2}$--invariant
representations. In order to show the $\Sigma_{p^2}$--invariance, we recall some facts from Section \ref{FusionSymmetric}
first. By Lemma \ref{conjS}, the $S$--conjugacy classes of $S$ fuse into the following $\Sigma_{p^2}$--conjugacy classes. 
For each $1 \leq m \leq p-1$, the elements
\[ \prod_{j=1}^m \sigma_{i_j}^{\alpha_j} \]
with $\alpha_j \in \{1,\ldots,p-1\}$ form the conjugacy class of products of $m$ disjoint $p$-cycles. The elements
\[ \prod_{j=1}^p \sigma_j^{\alpha_j} \tau^b , \]
where $\alpha_1+\ldots+\alpha_p$ or $b$ is a multiple of $p$, form the conjugacy class of products of $p$ disjoint $p$-cycles. The
elements
\[ \prod_{j=1}^p \sigma_j^{\alpha_j} \tau^b , \]
where both $\alpha_1+\ldots+\alpha_p$ and $b$ are not multiples of $p$, form the conjugacy class of $p^2$-cycles. And
of course we have the trivial conjugacy class.

\begin{lemma}
\label{CharacterX}
The characters of the representations $X_j$ satisfy
\[ \chi_{X_i}(\sigma_1 \cdots \sigma_p) = \left\{ \begin{array}{ll}
                                                   \binom{p}{i} (-1)^i & \text{ if } 1 \leq i < p , \\
                                                    & \\
                                                   -p & \text{ if } i= p .\end{array} \right.  \]
\end{lemma}

\begin{proof}
Assume first that $ 1 \leq i < p$. Then
\begin{align*}
\chi_{X_i}(\sigma_1 \cdots \sigma_p) & = \chi_{\Ind_K^S \left( \sum_{(j_0,\ldots,j_{p-1}) \in B_i} \omega^{j_0} \otimes \ldots \otimes \omega^{j_{p-1}} \right)}(\sigma_1 \cdots \sigma_p) \\
 & = \sum_{(j_0,\ldots,j_{p-1}) \in W_i} \chi_{\omega^{j_0} \otimes \ldots \otimes \omega^{j_{p-1}}}(\sigma_1 \cdots \sigma_p) \\
 & = \sum_{(j_0,\ldots,j_{p-1}) \in W_i} \omega^{j_0+\ldots+j_{p-1}}.
\end{align*}
Tuples in $W_i$ have exactly $i$ nonzero coordinates. Each $i$-tuple $(k_1,\ldots,k_i)$ in $\{1,\ldots,p-1\}^i$ appears $\binom{p}{i}$ times
as the nonzero coordinates of a tuple in $W_i$, hence the sum above equals
\[ \binom{p}{i} \sum_{(k_1,\ldots,k_i) \in \{1,\ldots,p-1\}^i} \omega^{k_1+\ldots+k_i}. \]
The number of solutions of $k_1+\ldots+k_i = 0$ is
\[ \sum_{j=1}^{i-1} (-1)^{j+1} (p-1)^{i-j}. \]
Each of these solutions contributes one to the sum above. On the other hand, the number of solutions to $k_1+\ldots+k_i \neq 0$ is
\[ (p-1)^i - \sum_{j=1}^{i-1} (-1)^{j+1} (p-1)^{i-j}. \]
There is a bijection between the tuples $(k_1,\ldots,k_i)$ with $k_1+\ldots+k_i = 1$ and the tuples $(r_1,\ldots,r_i)$
with $r_1+\ldots+r_i = \xi$ for any $\xi \in (\Z/p)^{\times}$. Hence for each $\xi$, the number of solutions to $k_1+\ldots+k_i = \xi$ is
\[ (p-1)^{i-1} - \sum_{j=1}^{i-1} (-1)^{j+1} (p-1)^{i-j-1} = \sum_{j=1}^i (-1)^{j+1} (p-1)^{i-j}. \]
Each of these solutions can be paired with a solution for $\xi^2$, a solution for $\xi^3$, and so on, hence all these solutions together
contribute 
\[ - \sum_{j=1}^i (-1)^{j+1} (p-1)^{i-j} \]
to the sum above. Therefore
\begin{align*}
\chi_{X_i}(\sigma_1 \cdots \sigma_p) & = \binom{p}{i} \left( \sum_{j=1}^{i-1} (-1)^{j+1} (p-1)^{i-j} - \sum_{j=1}^i (-1)^{j+1} (p-1)^{i-j} \right) \\
 & = - (-1)^{i+1} \binom{p}{i} \\
 & = (-1)^i \binom{p}{i} .
\end{align*}
as we wanted to prove. For $X_p$, we use a similar argument. 
\begin{align*}
\chi_{X_p}(\sigma_1 \cdots \sigma_p) & = \chi_{\Ind_K^S \left( \sum_{(j_0,\ldots,j_{p-1}) \in B_p} \omega^{j_0} \otimes \ldots \otimes \omega^{j_{p-1}} \right)}(\sigma_1 \cdots \sigma_p) \\
 & = \sum_{(j_0,\ldots,j_{p-1}) \in W_p} \chi_{\omega^{j_0} \otimes \ldots \otimes \omega^{j_{p-1}}}(\sigma_1 \cdots \sigma_p) \\
 & = \sum_{(j_0,\ldots,j_{p-1}) \in W_p} \omega^{j_0+\ldots+j_{p-1}} \\
 & = \sum_{(j_0,\ldots,j_{p-1}) \in \{1,\ldots,p-1\}^p-\Delta} \omega^{j_0+\ldots+j_{p-1}} ,
\end{align*}
where $\Delta$ is the diagonal in $\{1,\ldots,p-1\}^p$. In this case the number of solutions of $j_0+\ldots+j_{p-1} = 0$ is
\[ -(p-1) + \sum_{j=1}^{p-1} (-1)^{j+1} (p-1)^{p-j},  \]
while the rest of summands contribute
\[ -\sum_{j=1}^p (-1)^{j+1} (p-1)^{p-j}. \]
Therefore
\[ \chi_{X_p}(\sigma_1 \cdots \sigma_p) = -(p-1)-(-1)^{p+1} = -p ,\]
concluding the proof.
\end{proof}

\begin{lemma}
\label{SubrepsDeX}
Let $J \subseteq \{1,\ldots,p-1\}$ and let $V$ be a subrepresentation of $\sum_{j \in J} X_j$ such
that $\Res_K^S(V)$ is $\Aut_{\Sigma_{p^2}}(K)$--invariant. Then there exists $J' \subseteq J$ such
that $ V = \sum_{j \in J'} X_j$.
\end{lemma}

\begin{proof}
Note that 
\[ \Res_K^S(V) = \sum_{j \in J} \sum_{(j_0,\ldots,j_{p-1}) \in R_j \subseteq W_j} \omega^{j_0} \otimes \ldots \otimes \omega^{j_{p-1}} \]
for certain (possibly empty) subsets $R_j$. Since $\Aut_{\Sigma_{p^2}})K)$ contains the automorphisms that
permute the elements $\sigma_i$ and the automorphism that sends $\sigma_i$ to $\sigma_i^j$ for each $i$
and each $j \in \{1,\ldots,p-1\}$, the set $R_j$ can only be $W_j$ or the empty set.
\end{proof}

\begin{lemma}
\label{RepA}
The representation $ A = a_1^0 + \ldots + a_{p-1}^0 + X_1 $ is an irreducible $\Sigma_{p^2}$--invariant
representation.
\end{lemma}

\begin{proof}
We compute the character of $A$ on nontrivial elements of $S$. Let $\alpha_j \in \{1, \ldots, p-1 \}$ for $j=1,\ldots,m$
with $m \leq p$.
\begin{align*} 
\chi_A \left( \prod_{j=1}^m \sigma_{i_j}^{\alpha_j} \right) & = (p-1)+ \sum_{k=1}^{p-1} \Ind_K^S \chi_{\omega^k \otimes 1 \otimes \ldots \otimes 1} \left( \prod_{j=1}^m \sigma_{i_j}^{\alpha_j} \right) \\
          & = (p-1) + \sum_{k=1}^{p-1} [ \omega^{\alpha_1 k} + \ldots + \omega^{\alpha_m k} + (p-m) ] \\
          & = (p-1)+(p-1)(p-m)+ \sum_{k=1}^{p-1} m \omega^k \\
          & = (p-1)+(p-1)(p-m)-m.
\end{align*}
We see that the result only depends on $m$, not on the exponents $\alpha_j$. In particular, note that this equals $-1$ when $m=p$. Now let $\alpha_j \in \{0,\ldots,p-1 \}$
and $b \in \{1,\ldots,p-1\}$. 
\begin{align*} 
\chi_A \left( \prod_{j=1}^{p-1} \sigma_j^{\alpha_j} \tau^b \right) & = \chi_{a_1^0+\ldots+a_{p-1}^0} \left( \prod_{j=1}^{p-1} \sigma_j^{\alpha_j} \tau^b \right)  \\
          & = \omega+\ldots+\omega^{p-1} \\
          & = -1.
\end{align*}
Hence $A$ is $\Sigma_{p^2}$--invariant. 

Let $V$ be an irreducible $\Sigma_{p^2}$--invariant subrepresentation of $A$. As a representation of 
$S$, it must have the form $ V' + V''$, where $V'$ is a sum of representations of the form $a_j^0$ and $V''$ is a sum of representations
of the form $ \Ind_K^S(\omega^j \otimes 1 \otimes \ldots \otimes 1)$. Since the character of $V''$ vanishes on $S-K$, we have $\chi_{V'}(\tau)=\chi_{V'}(\tau^j)$
for all $j$. That is, $V'$ restricted to the subgroup generated by $\tau$ is $\Sigma_p$--invariant, hence $V'=0$ or $V'=a_1^0+\ldots+a_{p-1}^0$. Since $\Res_K^S(V')$
is a trivial representation, the representation $V''$ is $\Aut_{\Sigma_{p^2}}(K)$--invariant. By Lemma \ref{SubrepsDeX}, either
$V''=0$ or $V''=X_1$. Now we discard two of the possibilities for $V$. The representation $ a_1^0+\ldots+a_{p-1}^0$ is not $\Sigma_{p^2}$--invariant because 
\[ \chi_{a_1^0+\ldots+a_{p-1}^0}(\tau) = \sum_{j=1}^{p-1} \omega^j = -1 \neq p-1 = \chi_{a_1^0+\ldots+a_{p-1}^0}(\sigma_0 \cdots \sigma_{p-1}) . \]
The representation $X_1$ is not  $\Sigma_{p^2}$--invariant because
\[ \chi_{X_1}(\tau) = 0 \neq -p = \chi_{X_1}(\sigma_0 \cdots \sigma_{p-1}). \]
Therefore $V=A$ or $V=0$, and so $A$ is an irreducible $\Sigma_{p^2}$--invariant representation.
\end{proof}

\begin{lemma}
The representation $ B' = a_0^1 + \ldots + a_0^{p-1} + X_2 + \ldots + X_{p-1} + X_p $ is $\Sigma_{p^2}$--invariant.
\end{lemma}

\begin{proof}
We compute the character of $B'$ on nontrivial elements of $S$. Let $\alpha_j \in \{1, \ldots, p-1 \}$ for $j=1,\ldots,m$
with $m \leq p$.
\begin{align*} 
\chi_{B'} \left( \prod_{j=1}^m \sigma_{i_j}^{\alpha_j} \right) & = \sum_{(k_1,\ldots,k_p) \in W_2 \cup \ldots \cup W_p} \chi_{\omega_1^{k_1} \otimes \ldots \otimes \omega_p^{k_p}} \left( \prod_{j=1}^m \sigma_{i_j}^{\alpha_j} \right) \\
          & = \chi_{\reg_K-1-X_1}\left( \prod_{j=1}^m \sigma_{i_j}^{\alpha_j} \right) \\
          & = -1 - (p-1)(p-m)+m .
\end{align*}
Note that this equals $p-1$ when $m=p$. On the other hand, given $\alpha_j \in \{0,\ldots,p-1 \}$ and $b \in \{1,\ldots,p-1\}$. 
\begin{align*} 
\chi_{B'} \left( \prod_{j=1}^{p-1} \sigma_j^{\alpha_j} \tau^b \right) & = \chi_{a_0^1+\ldots+a_0^{p-1}} \left( \prod_{j=1}^{p-1} \sigma_j^{\alpha_j} \tau^b \right)  \\
          & = \omega^{\alpha_0+\ldots+\alpha_{p-1}} + \omega^{2(\alpha_0+\ldots+\alpha_{p-1})}+\ldots+\omega^{(p-1)(\alpha_0+\ldots+\alpha_{p-1})} \\
          & = \left\{ \begin{array}{ll}
                       p-1 & \text{ if } \alpha_0 + \ldots + \alpha_{p-1} \equiv 0 \Mod p ,\\
                       -1 & \text{ otherwise. } \end{array} \right.  
\end{align*}
This shows that $B'$ is $\Sigma_{p^2}$--invariant.
\end{proof}

\begin{lemma}
The representation $X_j+X_{p-j}$ is an irreducible $\Sigma_{p^2}$--invariant
representation for each $1 \leq j \leq p-1$.
\end{lemma}

\begin{proof}
For $ 1 \leq j \leq p-1$, we have
\[ \Res^S_K(X_j) = \sum_{(j_0,\ldots,j_{p-1}) \in W_j} \omega^{j_0} \otimes \ldots \otimes \omega^{j_{p-1}}, \]
which is clearly $\Aut_{\Sigma_{p^2}}(K)$--invariant. By Lemma \ref{CharacterX}, we have
\[ \chi_{X_j+X_{p-j}}(\sigma_1 \cdots \sigma_p) = 0. \]
Since the characters of $X_j$ and $X_{p-j}$ vanish on $S-K$, we conclude that $X_j+X_{p-j}$ is $\Sigma_{p^2}$--invariant.

Let $V$ be an irreducible $\Sigma_{p^2}$--invariant subrepresentation of $X_j+X_{p-j}$. By Lemma \ref{SubrepsDeX}, $V$ can only be 
$0$, $X_j$, $X_{p-j}$ or $X_j + X_{p-j}$. But $X_j$ and $X_{p-j}$ are not $\Sigma_{p^2}$--invariant because
their characters on $\sigma_1 \cdots \sigma_p$ are nonzero, but zero on $\tau$. Hence $V=X_j+X_{p-j}$ or $V=0$, and so $X_j+X_{p-j}$
is an irreducible $\Sigma_{p^2}$--invariant representation.
\end{proof}

We let $B = a_0^1 + \ldots + a_0^{p-1} + X_{p-1} + X_p$ and $C = X_1 + X_{p-1}$. By the previous lemma, $C$ is an irreducible
$\Sigma_{p^2}$--invariant representation. By the previous two lemmas, $B$ is $\Sigma_{p^2}$--invariant and we will show now that
it is irreducible.

\begin{lemma}
The representation $B = a_0^1 + \ldots + a_0^{p-1} + X_{p-1} + X_p$ is an irreducible $\Sigma_{p^2}$--invariant representation.
\end{lemma}

\begin{proof}
Let $V$ be an irreducible $\Sigma_{p^2}$--invariant subrepresentation of $B$. As a representation of $S$, it must have the form
$V'+V''$, where $V'$ is a sum of representations of the form $a_0^j$ and $V''$ is a sum of representations of the form $\Ind_K^S(\omega^{j_0} \otimes \ldots \otimes \omega^{j_{p-1}})$ with $(j_0,\ldots,j_{p-1}) \in B_{p-1} \cup B_p $. Since the character of $V''$ vanishes on $S-K$, we have $\chi_{V'}(\tau)=\chi_{V'}(\tau^j)$ for all $j$. That is, $V'$ restricted to the subgroup generated by $\tau$ is $\Sigma_p$--invariant, hence $V'=0$ or $V'=a_1^0+\ldots+a_{p-1}^0$. 

If $V'=0$, then $V''$ can not contain any subrepresentation of $X_p$ because otherwise its restriction to $K$ would contain a subrepresentation
of the form $\omega^{j_0} \otimes \ldots \otimes \omega^{j_{p-1}}$ with all $j_i \in \{1,\ldots,p-1\}$ and since $\Res_K^S(V)$ is $\Aut_{\Sigma_{p^2}}(K)$--invariant,
$V'$ would contain $ \omega \otimes \ldots \otimes \omega$. Therefore $V''=X_{p-1}$ or $V''=0$. But $X_{p-1}$ is not $\Sigma_{p^2}$--invariant because
\[ \chi_{X_{p-1}}(\tau) = 0 \neq p = \chi_{X_{p-1}}(\sigma_1 \cdots \sigma_p) . \]
Thus $V=0$. If $V' = a_1^0+\ldots+a_{p-1}^0$, then $V'$ must contain at least $X_p$ by the same argument above. But $ W = a_1^0+\ldots+a_{p-1}^0+X_p$ is not $\Sigma_{p^2}$--invariant
because
\[ \chi_W(\tau) = p-1 \neq p-1-p = \chi_W(\sigma_1 \cdots \sigma_p). \]
Therefore $V'' \neq 0$. Since $\Res_K^S(W)$ is $\Aut_{\Sigma_{p^2}}(K)$--invariant, so is $\Res_K^S(V'')$
and by Lemma \ref{SubrepsDeX}, we obtain $V''=X_{p-1}$. Therefore $V=B$ and so $B$ is an irreducible $\Sigma_{p^2}$--invariant representation.
\end{proof}

\begin{lemma}
The representation $ D = a_0^1+\ldots+a_0^{p-1}+a_1^0+\ldots+a_{p-1}^0+X_p$ is an irreducible $\Sigma_{p^2}$--invariant
representation.
\end{lemma}

\begin{proof}
Since $A+B=C+D$ and the previous lemmas show that $A$, $B$ and $C$ are $\Sigma_{p^2}$--invariant, so is $D$. Let $V$ be
an irreducible $\Sigma_{p^2}$--invariant subrepresentation of $D$. As a representation of $S$, it must have the form
$V'+V''$, where $V'$ is a subrepresentation of $a_0^1+\ldots+a_0^{p-1}+X_p$ and $V''$ is a subrepresentation of $a_1^0+\ldots+a_{p-1}^0$. 
Note that the restriction of $V$ to $K$ is a sum of representations of $\omega^{j_0} \otimes \ldots \otimes \omega^{j_{p-1}}$
with $j_i \in \{1,\ldots,p-1\}$, and trivial representations. Since $\Res_K^S(V)$ is $\Aut_{\Sigma_{p^2}}(K)$--invariant,
we must have $V'=0$ or $V' = a_0^1+\ldots+a_0^{p-1}+X_p$. We can not have $V'=0$ unless $V=0$ because we showed in the proof of Lemma \ref{RepA}
that no nontrivial subrepresentation of $a_1^0+\ldots+a_{p-1}^0$ is $\Sigma_{p^2}$--invariant. Hence $V'= a_0^1+\ldots+a_0^{p-1}+X_p$.
The restriction of $V'$ to the subgroup generated by $\tau$ is a sum of trivial representations and regular representations, while
the restriction of $V''$ is a sum of representations of the form $\omega^j$ with $j \neq 0$. Since the restriction of $V$ to the
subgroup generated by $\tau$ must be $\Sigma_p$--invariant, either $V''=0$ or $V''= a_1^0+\ldots+a_{p-1}^0$. But we saw in the
proof of the previous lemma that $a_0^1+\ldots+a_0^{p-1}+X_p$ is not $\Sigma_{p^2}$--invariant, hence $V=D$ and so $D$ is an 
irreducible $\Sigma_{p^2}$--invariant representation.
\end{proof}

These lemmas and the equality $A+B=C+D$ show the lack of uniqueness of decomposition of $\Sigma_{p^2}$--invariant
representations.

\begin{theorem}
If $p$ is an odd prime, the $\Sigma_{p^2}$--invariant representations of $\Z/p \wr \Z/p$ do not 
satisfy uniqueness of factorization as a sum of irreducible $\Sigma_{p^2}$--invariant representations.
\end{theorem}

\section{The case of $\Sigma_8$}
\label{Sigma8}

In this section we find a nonunique factorization into irreducible $\Sigma_8$--invariant representations
of $D_8 \wr \Z/2$.
\newline

We consider the $2$-Sylow subgroup $S$ of $\Sigma_8$ that is generated by the elements 
\[(1,3,2,4),(1,2),(5,7,6,8),(5,6),(1,5)(2,6)(3,7)(4,8), \]
which will be denoted by $r_1,s_1,r_2,s_2$ and $t$, respectively. Let $H_j$ be the subgroup generated by $r_j$ and $s_j$ for $j=1$, $2$. 
We also let $K$ be the subgroup generated by $t$. Notice that $H_j$ is isomorphic to $D_8$ by identifying $r_j$ and $s_j$ with the usual 
generators $r$ and $s$. We also have that $K$ is isomorphic to $\Z/2$. Note that $S$ is isomorphic to $D_8 \wr \Z/2$. There is a general
theory to determine the conjugacy classes of groups of the form $G \wr \Sigma_n$, see for instance Section 4.2 of \cite{JK}. However, since
$D_8 \wr \Z/2$ is small, we prefer to give a direct determination.

First note that the $D_8 \wr \Z/2$--conjugacy classes of elements in the normal subgroup $D_8 \times D_8$ 
correspond to the orbits under the action of $\Z/2$ on the $D_8 \times D_8$--conjugacy classes of $D_8 \times D_8$.
We obtain in this way the representatives
\[ \{ 1 ,  s_1 , r_2^2 , s_2r_2 , r_2 , s_1s_2 , s_1r_2^2 , s_1s_2r_2 , s_1r_2 , r_1^2r_2^2 , r_1^2 s_2 r_2 , r_1^2 r_2 , s_1r_1s_2r_2 , s_1r_1r_2 , r_1r_2 \}. \] 
On the other hand, an element which is not in $D_8 \times D_8$ must have the form $t(a,b)$ with $a$, $b \in D_8$. By conjugating
with $(a,1)$, we can assume that it has the form $t(1,x)$. It is easy to show that $t(1,x)$ is conjugate to $t(1,y)$ in $D_8 \wr \Z/2$ 
if and only if $x$ is conjugate to $y$ in $D_8$, therefore the obtain the desired remaining classes, represented by the elements
\[ \{t , ts_2 , tr_2^2 , ts_2r_2 , tr_2 \}. \]
Now that we found all the conjugacy classes of elements of $D_8 \wr \Z/2$, we can use their decomposition in disjoint cycles to find out 
how these classes fuse into ten $\Sigma_8$--conjugacy classes. First, the conjugacy classes of $1$, $s_1$, $r_2$, $s_1 r_2$ and $tr_2$ in 
$D_8 \wr \Z/2$ and $\Sigma_8$ coincide. The rest of conjugacy classes are fused in $\Sigma_8$ according to the following sets, where two 
elements belong to the same set if and only if they are $\Sigma_8$--conjugate.
\[ \{ r_2^2 , s_2r_2, s_1s_2 \}, \{ s_1r_2^2 , s_1s_2r_2 \}, \{ r_1^2r_2^2 , r_1^2 s_2 r_2, s_1r_1s_2r_2, t \}, \{ r_1^2 r_2, s_1r_1r_2, ts_2 \}, \{ r_1r_2, tr_2^2 , ts_2r_2 \}. \]

Next we describe the irreducible representations of $S\cong D_8\wr \Z/2$ using Theorem 25.6 from \cite{HCTFG}. As in Section \ref{Sigma4y9}, let 
$1$, $X$, $Y$, $XY$ and $Z$ denote the irreducible representations of $D_8$. Let $1$ and $\omega$ denote the
trivial and sign irreducible representations of $\Z/2$, respectively. For each irreducible representation $V$ of $D_8$, the character $\chi_V \chi_V$ of
$H_1 H_2$ is fixed under the action of $S$, hence it has an extension $\psi_V$ to $S$ and we obtain characters 
\[\psi_V \chi_W \]
of irreducible representations of $S$ for each irreducible representation $V$ of $D_8$ and each irreducible representation $W$ of $\Z/2$. We denote
by $ V \otimes V \otimes W$ the representation with character $\psi_V \chi_W$ for simplicity. Given two different irreducible representations $V$ and $V'$ of $D_8$, the isotropy of the
action of $S$ on $\chi_V \chi_{V'}$ equals $H_1 H_2$, hence $\Ind_{H_1H_2}^S(\chi_V \chi_{V'})$ is the character of an irreducible representation
of $S$. We denote this representation by $\Ind(V \otimes V')$. We use the following short names for the irreducible representations
of $S$. 

\begin{table}[H]
\begin{center}
	\begin{tabular}{l|l}
 &   \\ \hline
$x_1$ & $1$ \\
$x_2$ & $1\otimes 1\otimes \omega$ \\
$x_3$ & $X\otimes X\otimes 1$ \\
$x_4$ & $X\otimes X\otimes \omega$ \\
$x_5$ & $XY\otimes XY\otimes 1$  \\
$x_6$ & $XY\otimes XY\otimes \omega$ \\
$x_7$ & $Y\otimes Y\otimes 1$ \\
$x_8$ & $Y\otimes Y\otimes \omega$ \\
$x_9$ & $\Ind (1\otimes X)$ \\
$x_{10}$ & $\Ind (Y\otimes XY)$ \\
$x_{11}$ & $\Ind (1\otimes XY)$ \\
$x_{12}$ & $\Ind (X\otimes Y)$ \\
$x_{13}$ & $\Ind (1\otimes Y)$ \\
$x_{14}$ & $\Ind (X\otimes XY)$ \\
$x_{15}$ & $Z\otimes Z\otimes 1$ \\
$x_{16}$ & $Z\otimes Z\otimes \omega$ \\ 
$x_{17}$ & $\Ind (1\otimes Z)$ \\
$x_{18}$ & $\Ind (X\otimes Z)$ \\
$x_{19}$ & $\Ind (XY\otimes Z)$ \\
$x_{20}$ & $\Ind (Y\otimes Z)$ 
\end{tabular}
\caption{Short names for irreducible representations of $D_8 \wr \Z/2$}
\end{center}
\end{table}

We will include the character table for $S$ in the appendix. Using this table and the description
of how the conjugacy classes of $D_8 \wr \Z/2$ fused into $\Sigma_8$--conjugacy classes, it is easy
to check that the representations
\begin{align*}
A & = x_4 + x_9 + +x_{13}+x_{16}+ 2x_{17}+x_{18}, \\
B & = x_{12}+x_{14}+x_{15}+x_{16}+x_{18}+x_{19}+x_{20}, \\
C & = x_4+x_9+x_{14}+x_{16}+x_{17}+x_{18}+x_{19}, \\
D & = x_{12}+x_{13}+x_{15}+x_{16}+x_{17}+x_{18}+x_{20},
\end{align*}
are $\Sigma_8$--invariant.

\begin{proposition}
The representations $A$ and $B$ are irreducible $\Sigma_8$--invariant representations.
\end{proposition}

\begin{proof}
We first show that $A$ is irreducible as a $\Sigma_8$--invariant representation. Let $W$ be a
$\Sigma_8$--invariant subrepresentation of $A$. 

Assume first that $x_4$ is not a subrepresentation of $W$ so that $\chi_W(ts_2r_2) = 0$. Since 
$W$ is $\Sigma_8$--invariant, we must also have $\chi_W(tr_2^2)=0$ and therefore $x_{16}$ is not a 
subrepresentation of $W$. Then $\chi_W(t)=0$ and therefore $\chi_W(s_1r_1s_2r_2) =0$. Since the
characters of $x_{17}$ and $x_{18}$ vanish at $s_1r_1s_2r_2$ and 
\[ \chi_{x_9}(s_1r_1s_2r_2) = -2 = \chi_{x_{13}}(s_1r_1s_2r_2), \]
we conclude that $x_9$ and $x_{13}$ are not subrepresentations of $W$. Because $t$ and $r_1^2r_2^2$
are $\Sigma_8$--conjugate, we have $\chi_W(r_1^2r_2^2)=0$. But
\[ \chi_{x_{17}}(r_1^2r_2^2) = -4 = \chi_{x_{18}}(r_1^2r_2^2), \]
hence $x_{17}$ and $x_{18}$ are not subrepresentations of $W$, from where $W=0$.

Now assume that $x_4$ is a subrepresentation of $W$. Then $\chi_W(ts_2r_2)=1$ and therefore
$\chi_W(tr_2^2)=1$. Since the characters of $x_9$, $x_{13}$, $x_{17}$ and $x_{18}$ vanish on
$tr_2^2$ and 
\[ \chi_{x_4}(tr_2^2) = -1, \qquad \chi_{x_{16}}(tr_2^2) = 2, \]
we can assure that $W$ contains $x_{16}$ as a subrepresentation. Note that $\chi_W(t)=-3$ and
this must equal $\chi_W(s_1r_1s_2r_2)$. Since the characters of $x_{16}$, $x_{17}$ and $x_{18}$
vanish on $s_1r_1s_2r_2$ and
\[ \chi_{x_4}(s_1r_1s_2r_2) = 1, \qquad \chi_{x_9}(s_1r_1s_2r_2) = -2, \qquad \chi_{x_{13}}(s_1r_1s_2r_2) = -2, \]
we obtain that $x_9$ and $x_{13}$ are subrepresentations of $W$. Finally, because $\chi_W(t)=-3$,
we must have $\chi_W(r_1^2r_2^2)=-3$. We already know that $ V = x_4+x_9+x_{13}+x_{16}$ is a subrepresentation
of $W$ and
\[ \chi_V(r_1^2r_2^2) = 9, \qquad \chi_{x_{17}}(r_1^2r_2^2) = -4, \qquad \chi_{x_{18}}(r_1^2r_2^2) = -4, \]
we must have that $2x_{17}+x_{18}$ is a subrepresentation of $W$ and therefore $W=A$. Hence $A$ is
irreducible as a $\Sigma_8$--invariant representation.

Now we prove that $B$ is irreducible as a $\Sigma_8$--invariant representation. Let $W$ be a $\Sigma_8$--invariant
subrepresentation of $B$ and note that $\chi_W(tr_2s_2)=0$, hence $\chi_W(r_1r_2)=0$. Since the characters of the
irreducible subrepresentations of $B$ vanish on $r_1r_2$, except for $x_{12}$ and $x_{14}$ and
\[ \chi_{x_{12}}(r_1r_2) = -2, \qquad \chi_{x_{14}}(r_1r_2) = 2, \]
then either $x_{12}+x_{14}$ is a subrepresentation of $W$ or $W$ does not contain $x_{12}$ nor $x_{14}$ as subrepresentations.

Assume first that $x_{14}$ and $x_{12}$ are not subrepresentations of $W$. Then $\chi_W(s_1s_2)=0$ and therefore $\chi_W(r_2^2)=0$.
The characters of $x_{18}$, $x_{19}$ and $x_{20}$ vanish on $r_2^2$ and 
\[ \chi_{x_{15}}(r_2^2) = -4 = \chi_{x_{16}}(r_2^2), \]
hence $x_{15}$ and $x_{16}$ are not subrepresentations of $W$. Then $\chi_W(t)=0$ and therefore $\chi_W(r_1^2r_2^2) = 0$. The characters
of $x_{18}$, $x_{19}$ and $x_{20}$ equal $-4$ at $r_1^2r_2^2$, so we can conclude $W=0$.

Now assume that $x_{12}+x_{14}$ is a subrepresentation of $W$. Then $\chi_W(s_1s_2)=-4$ and therefore $\chi_W(r_2^2)=-4$. The characters
of $x_{18}$, $x_{19}$ and $x_{20}$ vanish on $r_2^2$ and
\[ \chi_{x_{12}+x_{14}}(r_2^2) = 4, \qquad \chi_{x_{15}}(r_2^2) = -4 = \chi_{x_{16}}(r_2^2), \]
hence $x_{15}+x_{16}$ must be a subrepresentation of $W$. Then $\chi_W(t)=0$ and therefore $\chi_W(r_1^2r_2^2)=0$. We have that 
$V=x_{12}+x_{14}+x_{15}+x_{16}$ is a subrepresentation of $W$ and
\[ \chi_V(r_1^2r_2^2) = 12, \qquad \chi_{x_{18}}(r_1^2r_2^2) = -4 = \chi_{x_{19}}(r_1^2r_2^2) = \chi_{x_{20}}(r_1^2r_2^2), \]
so we obtain $W=B$. Thus $B$ is irreducible as a $\Sigma_8$--invariant representation.
\end{proof}

Now we have $A+B=C+D$, where $A$ and $B$ are irreducible $\Sigma_8$--invariant representations
and $C$ and $D$ are $\Sigma_8$--invariant representations. Note that $C$ and $D$ do not contain
neither $A$ nor $B$ as subrepresentation. Therefore, after factoring, the right-hand side will give
us a factorization of $A+B$ in terms of different irreducible $\Sigma_8$--invariant representations.

\begin{corollary}
The $\Sigma_8$--invariant representations of $D_8 \wr \Z/2$ do not 
satisfy uniqueness of factorization as a sum of irreducible $\Sigma_8$--invariant 
representations.
\end{corollary}

\begin{remark}
It can be shown in a similar manner that $C$ and $D$ are irreducible $\Sigma_8$--invariant representations, so in
fact $C+D$ is the additional factorization of $A+B$. 
\end{remark}

\begin{remark}
We could not have used the irreducibility criterion from Remark 2.12 in \cite{CCo} here because we did not determine $R_{\Sigma_8}(D_8 \wr \Z/2)$. For the
same reason, it could not be used for the case of $\Sigma_{p^2}$ with $p>5$ from Section \ref{NonUnique}.
\end{remark}

\appendix

\section{Character tables}

In this appendix we show the character tables of $\Z/3 \wr \Z/3$ and $D_8 \wr \Z/2$ following the notations of Section \ref{Sigma4y9}
and Section \ref{Sigma8}.
\begin{table}[H]
\begin{center}
{\footnotesize
\noindent\adjustbox{max width=\textwidth}{   \begin{tabular}{l|l*{17}{c}r}
  & $1$ &  $\sigma_0$ & $\sigma_0^2$ & $\sigma_0 \sigma_1$ & $\sigma_0 \sigma_1^2$ & $\sigma_0^2 \sigma_1$ & $\sigma_0^2 \sigma_1^2$ & $\sigma_0 \sigma_1 \sigma_2$ & $\sigma_0 \sigma_1 \sigma_2^2$ & $\sigma_0 \sigma_1^2 \sigma_2^2$ & $\sigma_0^2 \sigma_1^2 \sigma_2^2$ & $\tau$ & $\sigma_0 \tau$ & $\sigma_0^2 \tau$ & $\tau^2$  & $\sigma_0 \tau^2 $ & $\sigma_0^2 \tau^2$ \\ \hline
$a_0$ & $1$ & $1$ & $1$ & $1$ & $1$ & $1$ & $1$ & $1$ & $1$ &$1$ &$1$ &$1$ &$1$ &$1$ &$1$ &$1$ &$1$  \\ 

 &  &  &  &  &  &  &  &  &  & & & & & & & &  \\

$a_1$       & $1$ & $1$ & $1$ & $1$ &  $1$& $1$& $1$& $1$& $1$ &$1$ &$1$ & $\omega$ & $\omega$ & $\omega$ & $\omega^2$ & $\omega^2$ & $\omega^2$   \\

 &  &  &  &  &  &  &  &  &  & & & & & & & &  \\

$a_2$	& $1$ & $1$ & $1$ & $1$ &  $1$& $1$& $1$& $1$& $1$ &$1$ &$1$ & $\omega^2$ & $\omega^2$ & $\omega^2$ & $\omega$ & $\omega$ & $\omega$ 	\\

 &  &  &  &  &  &  &  &  &  & & & & & & & &  \\

$b_0$	& $1$ & $\omega$ & $\omega^2$ & $\omega^2$ & $1$ & $1$ & $\omega$ & $1$ & $\omega$ & $\omega^2$ & $1$ & $1$ & $\omega$ & $\omega^2$ & $1$ & $\omega$ & $\omega^2$ 	 \\

 &  &  &  &  &  &  &  &  &  & & & & & & & &  \\

$b_1$	& $1$ & $\omega$ & $\omega^2$ & $\omega^2$ & $1$ & $1$ & $\omega$ & $1$ & $\omega$ & $\omega^2$ & $1$ & $\omega$ & $\omega^2$ & $1$ & $\omega^2$ & $1$ & $\omega$ 	 \\

 &  &  &  &  &  &  &  &  &  & & & & & & & &  \\

$b_2$	& $1$ & $\omega$ & $\omega^2$ & $\omega^2$ & $1$ & $1$ & $\omega$ & $1$ & $\omega$ & $\omega^2$ & $1$ & $\omega^2$ & $1$ & $\omega$ & $\omega$ & $\omega^2$ & $1$ 	 \\

 &  &  &  &  &  &  &  &  &  & & & & & & & &  \\

$c_0$		& $1$ & $\omega^2$ & $\omega$ & $\omega$ & $1$ & $1$ & $\omega^2$ & $1$ & $\omega^2$ & $\omega$ & $1$ & $1$ & $\omega^2$ & $\omega$ & $1$ & $\omega^2$ & $\omega$ 	\\

 &  &  &  &  &  &  &  &  &  & & & & & & & &  \\

$c_1$		& $1$ & $\omega^2$ & $\omega$ & $\omega$ & $1$ & $1$ & $\omega^2$ & $1$ & $\omega^2$ & $\omega$ & $1$ & $\omega$ & $1$ & $\omega^2$ & $\omega^2$ & $\omega$ & $1$ 	\\

 &  &  &  &  &  &  &  &  &  & & & & & & & &  \\

$c_2$		& $1$ & $\omega^2$ & $\omega$ & $\omega$ & $1$ & $1$ & $\omega^2$ & $1$ & $\omega^2$ & $\omega$ & $1$ & $\omega^2$ & $\omega$ & $1$ & $\omega$ & $1$ & $\omega^2$ \\

 &  &  &  &  &  &  &  &  &  & & & & & & & &  \\

$x_0$ & $3$ & $\omega + 2$ & $\omega^2 +2$ & $2\omega + 1$ & $0$ & $0$& $2\omega^2 +1$ & $3\omega$ & $2\omega +\omega^2$ & $2\omega^2 +\omega$ & $3\omega^2$ & $0$ & $0$ & $0$ & $0$ & $0$ & $0$\\

 &  &  &  &  &  &  &  &  &  & & & & & & & &  \\

$x_1$ & $3$& $\omega^2 +2$ & $\omega +2$ & $2\omega^2 +1$ & $0$ & $0$ & $2\omega +1$ & $3\omega^2$ & $2\omega^2 + \omega$ & $2\omega +
 \omega^2$ & $3\omega$ & $0$ & $0$ & $0$ & $0$ & $0$ & $0$\\
 
  &  &  &  &  &  &  &  &  &  & & & & & & & &  \\

$y_0$ & $3$ & $2\omega +1$ & $2\omega^2 +1$ & $2\omega + \omega^2$ & $0$ & $0$ & $2\omega^2 + \omega$ & $3\omega^2$ & $\omega^2 +2$ & $\omega +2$ & $3\omega$ & $0$ & $0$ & $0$ & $0$ & $0$ & $0$\\

 &  &  &  &  &  &  &  &  &  & & & & & & & &  \\
 
$y_1$ & $3$ & $0$ & $0$ & $0$ & $3\omega^2$ & $3\omega$ & $0$ & $3$ & $0$ & $0$ & $3$ & $0$ & $0$ & $0$ & $0$ & $0$ & $0$\\

 &  &  &  &  &  &  &  &  &  & & & & & & & &  \\
 
$y_2$ & $3$ & $0$ & $0$ & $0$ & $3\omega$ & $3\omega^2$ & $0$ & $3$ & $0$ & $0$ & $3$ & $0$ & $0$ & $0$ & $0$ & $0$ & $0$\\

 &  &  &  &  &  &  &  &  &  & & & & & & & &  \\
 
$y_3$ & $3$ & $2\omega^2 +1$ & $2\omega + 1$ & $2\omega^2 + \omega$ & $0$ & $0$ & $2\omega + \omega^2$ & $3\omega$ & $\omega +2$ & $\omega^2 +2$ & $3\omega^2$ & $0$ & $0$ & $0$ & $0$ & $0$ & $0$\\

 &  &  &  &  &  &  &  &  &  & & & & & & & &  \\
 
$z_0$ & $3$ & $2\omega + \omega^2$ & $ 2\omega ^2 + \omega$ & $\omega^2 +2$ & $0$ & $0$ & $\omega + 2$ & $3\omega$ & $2\omega^2 +1$ & $2\omega +1$ & $3\omega^2$ & $0$ & $0$ & $0$ & $0$ & $0$ & $0$\\

 &  &  &  &  &  &  &  &  &  & & & & & & & &  \\
 
$z_1$ & $3$ & $2\omega^2 +\omega$ & $2\omega +\omega^2$ & $\omega +2$ & $0$ & $0$ & $\omega^2 +2$ &  $3\omega^2$ & $2\omega + 1$ & $2\omega^2 +1$ & $3\omega$ & $0$ & $0$ & $0$ & $0$ & $0$ & $0$
	\end{tabular} }} \caption{Character table of $\Z/3 \wr \Z/3$.} 
\end{center}
\end{table} 

\begin{table}[H]
\begin{center}
{\footnotesize
\noindent\adjustbox{max width=\textwidth}{
  \begin{tabular}{l|l*{20}{c}r}
 & $1$ &  $s_1$ & $r_2^2$ & $s_2r_2$ & $r_2$ & $s_1s_2$ & $s_1r_2^2$ & $s_1s_2r_2$ & $s_1r_2$ & $r_1^2r_2^2$ & $r_1^2 s_2 r_2$ & $r_1^2 r_2$ & $s_1r_1s_2r_2$ & $s_1r_1r_2$ & $r_1r_2$  & $t$ & $ts_2$ & $tr_2^2$ & $ts_2r_2$ & $tr_2$ \\ \hline
$x_1$ & $1$ & $1$ & $1$ & $1$ & $1$ & $1$ & $1$ & $1$ & $1$ &$1$ &$1$ &$1$ &$1$ &$1$ &$1$ &$1$ &$1$ &$1$ &$1$ &$1$  \\ 

 &  &  &  &  &  &  &  &  &  & & & & & & & & & & & \\

$x_2$ & $1$ & $1$ & $1$ & $1$ &  $1$& $1$& $1$& $1$& $1$ &$1$ &$1$ & $1$ & $1$ & $1$ & $1$ & $-1$ & $-1$ & $-1$ & $-1$ & $-1$  \\

  &  &  &  &  &  &  &  &  &  & & & & & & & & & & & \\

$x_3$	& $1$ & $1$ & $1$ & $-1$ &  $-1$& $1$& $1$& $-1$& $-1$ &$1$ &$-1$ & $-1$ & $1$ & $1$ & $1$ & $1$ & $1$ & $1$ & $-1$ & $-1$ 	\\

  &  &  &  &  &  &  &  &  &  & & & & & & & & & & & \\

$x_4$	& $1$ & $1$ & $1$ & $-1$ &  $-1$& $1$& $1$& $-1$& $-1$ &$1$ &$-1$ & $-1$ & $1$ & $1$ & $1$ & $-1$ & $-1$ & $-1$ & $1$ & $1$ 	\\

  &  &  &  &  &  &  &  &  &  & & & & & & & & & & & \\

$x_5$	& $1$ & $-1$ & $1$ & $1$ &  $-1$& $1$& $-1$& $-1$& $1$ &$1$ &$1$ & $-1$ & $1$ & $-1$ & $1$ & $1$ & $-1$ & $1$ & $1$ & $-1$ 	\\

  &  &  &  &  &  &  &  &  &  & & & & & & & & & & & \\

$x_6$	& $1$ & $-1$ & $1$ & $1$ &  $-1$& $1$& $-1$& $-1$& $1$ &$1$ &$1$ & $-1$ & $1$ & $-1$ & $1$ & $-1$ & $1$ & $-1$ & $-1$ & $1$ 	\\

  &  &  &  &  &  &  &  &  &  & & & & & & & & & & & \\

$x_7$	& $1$ & $-1$ & $1$ & $-1$ &  $1$& $1$& $-1$& $1$& $-1$ &$1$ &$-1$ & $1$ & $1$ & $-1$ & $1$ & $1$ & $-1$ & $1$ & $-1$ & $1$ 	\\

  &  &  &  &  &  &  &  &  &  & & & & & & & & & & & \\

$x_8$	& $1$ & $-1$ & $1$ & $-1$ &  $1$& $1$& $-1$& $1$& $-1$ &$1$ &$-1$ & $1$ & $1$ & $-1$ & $1$ & $-1$ & $1$ & $-1$ & $1$ & $-1$ 	\\

  &  &  &  &  &  &  &  &  &  & & & & & & & & & & & \\

$x_9$	& $2$ & $2$ & $2$ & $0$ &  $0$& $2$& $2$& $0$& $0$ &$2$ &$0$ & $0$ & $-2$ & $-2$ & $-2$ & $0$ & $0$ & $0$ & $0$ & $0$ 	\\

  &  &  &  &  &  &  &  &  &  & & & & & & & & & & & \\

$x_{10}$ & $2$ & $-2$ & $2$ & $0$ &  $0$& $2$& $-2$& $0$& $0$ &$2$ &$0$ & $0$ & $-2$ & $2$ & $-2$ & $0$ & $0$ & $0$ & $0$ & $0$ 	\\

  &  &  &  &  &  &  &  &  &  & & & & & & & & & & & \\

$x_{11}$ & $2$ & $0$ & $2$ & $2$ &  $0$& $-2$& $0$& $0$& $-2$ &$2$ &$2$ & $0$ & $2$ & $0$ & $-2$ & $0$ & $0$ & $0$ & $0$ & $0$ 	\\
 
   &  &  &  &  &  &  &  &  &  & & & & & & & & & & & \\

$x_{12}$ & $2$ & $0$ & $2$ & $-2$ &  $0$& $-2$& $0$& $0$& $2$ &$2$ &$-2$ & $0$ & $2$ & $0$ & $-2$ & $0$ & $0$ & $0$ & $0$ & $0$ 	\\

  &  &  &  &  &  &  &  &  &  & & & & & & & & & & & \\
 
$x_{13}$ & $2$ & $0$ & $2$ & $0$ &  $2$& $-2$& $0$& $-2$& $0$ &$2$ &$0$ & $2$ & $-2$ & $0$ & $2$ & $0$ & $0$ & $0$ & $0$ & $0$ 	\\

  &  &  &  &  &  &  &  &  &  & & & & & & & & & & & \\
 
$x_{14}$ & $2$ & $0$ & $2$ & $0$ &  $-2$& $-2$& $0$& $2$& $0$ &$2$ &$0$ & $-2$ & $-2$ & $0$ & $2$ & $0$ & $0$ & $0$ & $0$ & $0$ 	\\

  &  &  &  &  &  &  &  &  &  & & & & & & & & & & & \\
 
$x_{15}$ & $4$ & $0$ & $-4$ & $0$ &  $0$& $0$& $0$& $0$& $0$ &$4$ &$0$ & $0$ & $0$ & $0$ & $0$ & $2$ & $0$ & $-2$ & $0$ & $0$ 	\\

  &  &  &  &  &  &  &  &  &  & & & & & & & & & & & \\
 
$x_{16}$ & $4$ & $0$ & $-4$ & $0$ &  $0$& $0$& $0$& $0$& $0$ &$4$ &$0$ & $0$ & $0$ & $0$ & $0$ & $-2$ & $0$ & $2$ & $0$ & $0$ 	\\

  &  &  &  &  &  &  &  &  &  & & & & & & & & & & & \\
 
$x_{17}$ & $4$ & $2$ & $0$ & $2$ &  $2$& $0$& $-2$& $0$& $0$ &$-4$ &$-2$ & $-2$ & $0$ & $0$ & $0$ & $0$ & $0$ & $0$ & $0$ & $0$ 	\\

  &  &  &  &  &  &  &  &  &  & & & & & & & & & & & \\
 
$x_{18}$ & $4$ & $2$ & $0$ & $-2$ &  $-2$& $0$& $-2$& $0$& $0$ &$-4$ &$2$ & $2$ & $0$ & $0$ & $0$ & $0$ & $0$ & $0$ & $0$ & $0$ 	\\

  &  &  &  &  &  &  &  &  &  & & & & & & & & & & & \\
 
$x_{19}$ & $4$ & $-2$ & $0$ & $2$ &  $-2$& $0$& $2$& $0$& $0$ &$-4$ &$-2$ & $2$ & $0$ & $0$ & $0$ & $0$ & $0$ & $0$ & $0$ & $0$ 	\\

  &  &  &  &  &  &  &  &  &  & & & & & & & & & & & \\
 
$x_{20}$ & $4$ & $-2$ & $0$ & $-2$ &  $2$& $0$& $2$& $0$& $0$ &$-4$ &$2$ & $-2$ & $0$ & $0$ & $0$ & $0$ & $0$ & $0$ & $0$ & $0$ 
	\end{tabular} }} \caption{Character table of $D_8 \wr \Z/2$.}
\end{center}
\end{table}

\noindent \begin{bf}Acknowledgments:\end{bf} The first author was partially supported by CONAHCYT's Frontier Science Project CF-2023-I-2649.



\begin{thebibliography}{10}

\bibitem{AD}
Jianbei An and Heiko Dietrich, \emph{The essential rank of fusion systems of
  blocks of symmetric groups}, Internat. J. Algebra Comput. \textbf{22} (2012),
  no.~1, 1250002, 15. \url{https://doi.org/10.1007/BF02698718}

\bibitem{A}
M.~F. Atiyah, \emph{Characters and cohomology of finite groups}, Inst. Hautes
  \'Etudes Sci. Publ. Math. (1961), no.~9, 23--64. \url{https://doi.org/10.1007/BF02698718}

\bibitem{AS}
M.~F. Atiyah and G.~B. Segal, \emph{Equivariant {$K$}-theory and completion},
  J. Differential Geometry \textbf{3} (1969), 1--18. \url{https://doi.org10.4310/jdg/1214428815}

\bibitem{BC}
No\'{e} B\'{a}rcenas and Jos\'{e} Cantarero, \emph{A completion theorem for
  fusion systems}, Israel J. Math. \textbf{236} (2020), no.~2, 501--531.
  \url{https://doi.org/10.1007/s11856-020-1981-4}

\bibitem{BLO}
Carles Broto, Ran Levi, and Bob Oliver, \emph{The homotopy theory of fusion
  systems}, J. Amer. Math. Soc. \textbf{16} (2003), no.~4, 779--856.
  \url{https://doi.org/10.1090/S0894-0347-03-00434-X}

\bibitem{CC}
Jos\'{e} Cantarero and Nat\`alia Castellana, \emph{Unitary embeddings of finite
  loop spaces}, Forum Math. \textbf{29} (2017), no.~2, 287--311. \url{https://doi.org/10.1515/forum-2015-0104}

\bibitem{CCM}
Jos\'{e} Cantarero, Nat\`alia Castellana, and Lola Morales, \emph{Vector
  bundles over classifying spaces of {$p$}-local finite groups and
  {B}enson-{C}arlson duality}, J. Lond. Math. Soc. (2) \textbf{101} (2020),
  no.~1, 1--22. \url{https://doi.org/10.1112/jlms.12255}

\bibitem{CCo}
Jos\'{e} Cantarero and Germ\'{a}n Combariza, \emph{Uniqueness of factorization for
  fusion-invariant representations}, 2023. \url{https://arxiv.org/abs/2303.10341}

\bibitem{CLl}
Humberto C\'{a}rdenas and Emilio Lluis, \emph{The normalizer of the {S}ylow
  {$p$}-group of the symmetric group {$S_{p^{n}}$}}, Bol. Soc. Mat. Mexicana
  (2) \textbf{9} (1964), 1--6.

\bibitem{G}
Jorge~E. Gaspar-Lara, \emph{K-teor{\'i}a y representaciones invariantes bajo
  fusi{\'o}n}, B.{S}c. thesis, Universidad Aut{\'o}noma de M{\'e}xico, 2020. \url{http://132.248.9.195/ptd2020/marzo/0801675/Index.html}

\bibitem{HCTFG}
Bertram Huppert, \emph{Character theory of finite groups}, De Gruyter
  Expositions in Mathematics, vol.~25, Walter de Gruyter \& Co., Berlin, 1998.
  \url{https://doi.org/10.1515/9783110809237}

\bibitem{JK}
Gordon James and Adalbert Kerber, \emph{The representation theory of the
  symmetric group}, Encyclopedia of Mathematics and its Applications, vol.~16,
  Addison-Wesley Publishing Co., Reading, Mass., 1981, With a foreword by P. M.
  Cohn, With an introduction by Gilbert de B. Robinson. \url{https://doi.org/10.1017/CBO9781107340732}

\bibitem{MGL}
Swati Maheshwari and R.~K. Sharma, \emph{A note on presentation of general
  linear groups over a finite field}, Southeast Asian Bull. Math. \textbf{43}
  (2019), no.~2, 217--224. \MR{3931967}

\bibitem{NS}
Toke Norgard-Sorensen, \emph{Homotopy representations of simply connected
  p-compact groups of rank 1 or 2}, Ph.{D}. thesis, University of Copenhagen,
  2013. \url{http://web.math.ku.dk/~jg/students/noergaard-soerensen.phd-thesis.pdf}

\bibitem{Re}
Sune~Precht Reeh, \emph{The abelian monoid of fusion-stable finite sets is
  free}, Algebra Number Theory \textbf{9} (2015), no.~10, 2303--2324.
  \url{https://doi.org/10.2140/ant.2015.9.2303}

\bibitem{SJLR}
Jean-Pierre Serre, \emph{Linear representations of finite groups},
  Springer-Verlag, New York-Heidelberg, 1977. \url{https://doi.org/10.1007/978-1-4684-9458-7}

\bibitem{Z1}
Krzysztof Ziemia\'{n}ski, \emph{Homotopy representations of {${\rm SO}(7)$} and
  {$\rm Spin(7)$} at the prime 2}, J. Pure Appl. Algebra \textbf{212} (2008),
  no.~6, 1525--1541. \url{https://doi.org/10.1016/j.jpaa.2007.10.018}

\bibitem{Z2}
Krzysztof Ziemia\'{n}ski, \emph{A faithful unitary representation of the 2-compact group {${\rm
  DI}(4)$}}, J. Pure Appl. Algebra \textbf{213} (2009), no.~7, 1239--1253.
  \url{https://doi.org/10.1016/j.jpaa.2008.11.042}

\end{thebibliography}

\end{document}